\documentclass[a4paper,11pt]{article}
\usepackage{amsmath,amsthm,amssymb,amsfonts}
\usepackage[svgnames,x11names,table]{xcolor}
\usepackage[utf8]{inputenc}
\usepackage{mathtools}
\usepackage{enumitem}

\usepackage{graphicx}
\usepackage{geometry}
\usepackage{framed}  
\usepackage[ruled,vlined]{algorithm2e}
\usepackage{booktabs}
\usepackage{cite}
\usepackage{hyperref} 
\hypersetup{
        colorlinks=true,   
        linkcolor=blue,    
        citecolor=green,    
        urlcolor=cyan      
      }
      
\usepackage{tikz}
\usetikzlibrary{positioning}

\usepackage{bbm}

\usepackage{xfrac}
\usepackage{microtype}
\renewcommand{\leq}{\leqslant}
\renewcommand{\geq}{\geqslant}
\renewcommand{\emptyset}{\varnothing}

\newtheorem{theorem}{Theorem}[section]
\newtheorem{lemma}[theorem]{Lemma}
\newtheorem{proposition}[theorem]{Proposition}
\newtheorem{corollary}[theorem]{Corollary}
\newtheorem{definition}[theorem]{Definition}
\newtheorem{assumption}[theorem]{Assumption}
\newtheorem{remark}[theorem]{Remark}
\newtheorem{example}[theorem]{Example}
\newtheorem{problem}[theorem]{Problem}

\newcommand{\R}{\mathbbm R}
\newcommand{\Rext}{(-\infty,+\infty]}
\newcommand{\N}{\mathbbm N}

\DeclareMathOperator{\prox}{prox}

\DeclareMathOperator{\dom}{dom}
\DeclareMathOperator{\gra}{gra}

\DeclareMathOperator{\range}{range}

\DeclareMathOperator{\ri}{\mathrm{ri}}

\newcommand{\Hilbert}{\mathcal{H}}
\newcommand{\G}{\mathcal{G}}

\newcommand{\setto}{\rightrightarrows}

\DeclareMathOperator{\Id}{Id}
\DeclareMathOperator*{\argmin}{arg\,min}

\newcommand{\bz}{\mathbf{z}}
\newcommand{\bx}{\mathbf{x}}
\newcommand{\by}{\mathbf{y}}
\newcommand{\bv}{\mathbf{v}}
\newcommand{\bu}{\mathbf{u}}
\newcommand{\bs}{\mathbf{s}}
\newcommand{\bw}{\mathbf{w}}

\newcommand{\hbx}{\widehat{\bx}}
\newcommand{\hby}{\widehat{\by}}

\newcommand{\cE}{\mathcal{E}}
\newcommand{\cG}{\mathcal{G}}
\newcommand{\cV}{\mathcal{V}}
\newcommand{\cR}{\mathcal{R}}

\usepackage{accents}
\newcommand\munderbar[1]{%
	\underaccent{\bar}{#1}}

\title{Primal-dual forward-reflected splitting with linesearch for distributed min-max problems under local Lipschitz smoothness}

\author{Felipe Atenas\thanks{Centro de Modelamiento Matem\'atico (CNRS IRL2807), Universidad de Chile, Santiago, \textsc{Chile}. 
		Email:~\href{mailto:fatenas@cmm.uchile.cl}
		{fatenas@cmm.uchile.cl}}
	\and
	Minh N. Dao\thanks{School of Science, RMIT University, Melbourne,
		                             \textsc{Australia}.
		Email:~\href{mailto:minh.dao@rmit.edu.au}
		{minh.dao@rmit.edu.au}}
	\and 
	Matthew K. Tam\thanks{School of Mathematics \& Statistics,
		The University of Melbourne,
		\textsc{Australia}.
		Email:~\href{mailto:matthew.tam@unimelb.edu.au}
		{matthew.tam@unimelb.edu.au}}
	\and
	David Torregrosa-Bel\'en\thanks{Department of Mathematics,
		University of Alicante,
		Alicante, \textsc{Spain}.
		Email:~\href{mailto:david.torregrosa@ua.es}
		{david.torregrosa@ua.es}}
}

\makeatletter
\renewcommand{\maketitle}{
	\begin{center}
		{\LARGE\bfseries\@title\par}
		\vskip 2em
		{\large
			\lineskip .5em
			\begin{tabular}[t]{c}
				\@author
			\end{tabular}\par}
	\end{center}
	\@thanks
}
\makeatother

\begin{document}
\maketitle

\begin{abstract}
We consider a network of finitely many agents aiming to cooperatively solve a convex-concave min-max problem combining nonsmooth terms with a smooth saddle function whose partial gradients are only locally Lipschitz continuous. Without requiring prior  knowledge of the local Lipschitz moduli of the partial gradients, we propose a distributed proximal-reflected-gradient algorithm which incorporates a distributed backtracking procedure to adjust the stepsizes at each iteration. The algorithm is derived by reformulating the min-max problem as a composite monotone inclusion involving the sum of a maximally monotone operator, a linearly composed maximally monotone operator, and a single-valued locally Lipschitz continuous monotone operator. To solve this inclusion, we introduce a novel primal-dual forward-reflected splitting method with linesearch and establish its convergence using a Lyapunov-based analysis. In addition, when specialised to the min-max formulation, we derive an $\mathcal{O}(1/N)$ ergodic convergence rate for a consensus-penalised gap function. Finally, applying this splitting framework with mixing matrices yields our distributed method for consensus min-max optimisation.
\end{abstract}

\paragraph{Keywords} distributed consensus optimisation, min-max problems, proximal algorithm, primal-dual splitting algorithm, backtracking linesearch
\paragraph{MSC2020}  68W15, 90C25, 	90C47,  47N10, 65K05

\section{Introduction}
In this note, we devise a first-order distributed algorithm for the following \emph{min-max problem}.

\begin{problem}\label{problem:1}
For all $i\,{\in}\,{\{1,2,\ldots,n\}}$, let the functions $f_i:\R^p\to{(-\infty,+\infty]}$ and $g_i:\R^q\to{(-\infty,+\infty]}$ be proper lower semicontinuous convex functions, and $\phi_i:\R^p\times\R^q\to\R$ be differentiable functions whose  partial gradients $\nabla_x\phi_i(\cdot,\cdot)$ and $\nabla_y\phi_i(\cdot,\cdot)$ are locally Lipschitz continuous on $\overline{\dom} \,f_i \times \overline\dom\,g_i$, and such that $\phi_i(\cdot,y)$ is convex on ${\dom}\, f_i$ for every $y\in {\dom}\, g_i$, and $\phi_i(x,\cdot)$ is concave on ${\dom}\, g_i$ for every $x\in {\dom}\, f_i$.
Find a solution of the min-max problem
\begin{equation}\label{eq:min-max}
	\min_{x\in\R^p}\max_{y\in\R^q} \;  \Psi(x,y):= \sum_{i=1}^n \bigl( f_i(x) + \phi_i(x,y)-g_i(y) \bigr),
\end{equation}
where solutions to~\eqref{eq:min-max} are understood in the sense of \emph{saddle-points}. Specifically, a point $(x^*,y^*)$ is said to be a saddle point of~\eqref{eq:min-max} if
\[
\Psi(x^*,y) \leq \Psi(x^*,y^*) \leq \Psi(x,y^*), \quad \forall x\in\R^p, \, \forall y \in\R^q.
\]
\end{problem}
For further details on min-max problems and sufficient conditions for the existence of saddle-points, we refer the interested reader to~\cite{MR274683}.

Works focused on designing decentralised schemes for saddle-point problems often address particular instances of Problem~\ref{problem:1} or require additional assumptions. Indeed, the authors in~\cite{9304470,NEURIPS2022_c959bb2c} make use of strong monotonicity assumptions, while in~\cite{MR4979666,MR4979661,MR4979664} the functions $f_i$ and $g_i$ are restricted to indicator functions of compact convex sets.
More recently, a distributed algorithm for min-max problems in the form of~\eqref{eq:min-max} was proposed by Malitsky and Tam~\cite{malitsky2023first} under the more restrictive assumption of the partial gradients of the functions $\phi_i$ being globally Lipschitz continuous. Following this work, we assume that a network of $n$ agents works cooperatively to solve Problem~\ref{problem:1} under the setup: (i) each agent $i$ computes the gradients of $\phi_i$ along with the proximity operators of $f_i$ and $g_i$, and (ii) agents exchange information via  their immediate neighbours in the network. Departing from~\cite{malitsky2023first}, our primary theoretical innovation lies in addressing the lack of global Lipschitz continuity in the gradients of $\phi_i$. Specifically, we assume: (iii) there exists no knowledge of the (local) Lipschitz constants of the gradients $\nabla_x\phi_i(\cdot,\cdot)$ and $\nabla_y \phi_i(\cdot,\cdot)$. Problems satisfying conditions (i)-(iii) arise, for instance, in decentralised approaches to distributionally robust optimisation~\cite{MR2486041,MR4379565} (see Section~\ref{sect:example} below for a concrete example illustrating this).

Algorithms belonging to the PG-EXTRA family~\cite{MR3343366,shi2015proximal} represent some of the most prominent schemes in distributed optimisation. It is well established that PG-EXTRA can be derived as a specific instance of a primal–dual splitting algorithm---namely, the Condat--V\~u algorithm~\cite{MR3037034,MR3084386}---applied to an appropriate monotone operator inclusion; see, e.g.,~\cite{malitsky2023first,MR4496334}. In this fashion, we derive our distributed algorithm by first reformulating the problem as a composite monotone inclusion problem of the following form.
\begin{problem}\label{problem:mi}
Let $\Hilbert$ and $\G$ be finite-dimensional real Hilbert spaces. Consider the inclusion problem given by
\begin{equation}\label{eq:monotone_inclusion}
	\text{find } z\in\Hilbert \text{ such that }  0\in\bigl(A+B\bigr)(z) + K^*C(Kz),
\end{equation}
where $A:\Hilbert\setto\Hilbert$, $C:\G\setto\G$ are maximally monotone operators, $B:\Hilbert\to\Hilbert$ is monotone 
on ${\dom}\, A$ and locally Lipschitz continuous on ${\overline{\dom}}\, A$, $K:\Hilbert\to\G$ is a  linear operator, and $K^*:\G\to\Hilbert$ its adjoint operator. 	
\end{problem}

It is well-known that dealing with monotone Lipschitz continuous operators requires considering different numerical schemes than the ones used for cocoercive operators; see, e.g.,~\cite{MR4666868,MR4101375,MR1741147,MR4093628}. In addition,  when the operator fails to be globally Lipschitz continuous or the Lipschitz constant cannot be accurately estimated, further efforts are required to select appropriate stepsizes; see, e.g.,~\cite{MR4101375,MR3860903,MR4490614}. 
Different splitting algorithms with linesearch have been proposed to handle monotone inclusions given by the sum of a  maximally monotone operator and a locally Lipschitz continuous  monotone operator; see, e.g.,~\cite{MR1741147,MR4101375,MR4490614,MR3860903}. These methods can solve Problem~\ref{problem:mi} by reformulating~\eqref{eq:monotone_inclusion} as the two-operator primal-dual inclusion
\[
\text{find } \begin{bmatrix} u \\ z
\end{bmatrix}\in\G\times\Hilbert \text{ such that } 0\in (\mathcal{A} + \mathcal{B})\begin{bmatrix} u \\ z
\end{bmatrix},
\]
where 
\[
\mathcal{A}:= \left[
\begin{array}{cc}
C^{-1} & 0 \\
0 & A
\end{array}
\right]
\quad
\text{and}
\quad
\mathcal{B}:=
\left[
\begin{array}{cc}
0 &  -K\\
K^* & B
\end{array}
\right].
\]
Specifically, they resort to forward-backward schemes that perform direct evaluations of the operator $\mathcal{B}$ and the resolvent of the set-valued operator $\mathcal{A}$. To illustrate this, we consider the \emph{forward-reflected-backward method with linesearch} in~\cite{MR4101375} that defines the iterates as
\[
\omega^{k+1} = J_{\gamma_k \mathcal{A}} \bigl(\omega^k - \gamma_k \mathcal{B}(\omega^k)- \gamma_{k-1} (\mathcal{B}(\omega^k)-\mathcal{B}(\omega^{k-1}) \bigr),
\]
with $\omega^{k}=(u^k,z^k)\in\cG\times\Hilbert$ for all $k\geq -1$, and where the stepsize $\gamma_k$ is taken as the largest number in $\{\lambda_{k-1},\rho\lambda_{k-1},\rho^2\lambda_{k-2},\ldots\}$ such that
\begin{equation}\label{eq:intro-lscond}
\gamma_k\|\mathcal{B}(\omega^{k+1})-\mathcal{B}(\omega^k)\| \leq \frac{\beta}{2}\|\omega^{k+1}-\omega^k\|,
\end{equation}
for $\rho,\beta\in(0,1)$. Since the linesearch condition~\eqref{eq:intro-lscond} requires the computation of  $\mathcal{B}(\omega^{k+1})$, every linesearch trial evaluates $B$, $K$, $K^*$ and the resolvents of both $A$ and $C^{-1}$. In this manner, an accepted linesearch iteration uses two evaluations of $B$, two of $K$ and two of $K^*$.

In this work, we develop a new \emph{primal-dual forward-reflected algorithm  with linesearch} (PDFRL) tailored to~Problem~\ref{problem:mi}  that alleviates the computational burden of the linesearch with respect to the aforementioned approaches. Specifically, PDFRL generates a pair of sequences converging to a primal-dual solution of Problem~\ref{problem:mi} by evaluating the resolvents of $A$ and $C^{-1}$ separately. The lack of cocoercivity of $B$ is resolved by incorporating reflected-type terms of this operator. Notably, PDFLR's linesearch only involves the sequence of primal variables generated by the algorithm, thereby it does not require  additional evaluations of the resolvent of $C^{-1}$. Furthermore, only one call to $K$ and $K^*$ are needed per iteration, independently of the number of linesearch trials performed. 
This, in turn, will permit us to derive a distributed algorithm with a single round of communication per iteration by specialising PDFRL to Problem~\ref{problem:1}.

It is worth noting that when the monotone operator $B$ is assumed to be globally Lipschitz, a fixed stepsize can be taken for all the iterations of our method, and no linesearch is needed. In such a case, our scheme specialises to the \emph{primal-dual twice reflected} algorithm  in~\cite{malitsky2023first} with extrapolation performed on the dual variable.  As explained in~\cite{malitsky2023first} (see also Remark~\ref{remark:particularcases} below), this method admits a representation as an instance of the forward-backward-reflected splitting~\cite{MR4101375} under a different metric, which is used to establish its convergence. While this suggests that it may be possible to analyze the convergence of PDFRL with linesearch using a similar (variable) metric reformulation, it is unclear whether this would lead to a simpler presentation of the result. We therefore establish convergence to a primal-dual solution using a direct Lyapunov-based analysis. This also allows us to obtain ergodic convergence rates of $\mathcal{O}(1/N)$, where $N$ is the number of iterations of the algorithm, for a  consensus-penalised gap function associated to Problem~\ref{problem:1} (see~\eqref{eq:gapfunction} below).

 We derive our proposed \emph{distributed proximal-reflected-gradient method with linesearch} (DPRGL) for solving Problem~\ref{problem:1}, by appropriately applying PDFRL using mixing matrices (see Definition~\ref{def:mm} below). In particular, devising DPRGL required tailoring the linesearch step in PDFRL to agent-wise operations. Recently,~\cite{atenas2024distributed} and~\cite{chen2025parameter}  have proposed distributed linesearch mechanisms for proximal-gradient methods for minimisation problems. Here,  the framework developed in~\cite{atenas2024distributed} is extended to the min-max setting to equip our algorithm with a backtracking linesearch that relies on a global scalar reduction across all agents.

The remainder of this work is organised as follows. Section~\ref{sect:preliminaries} introduces some preliminary notation. Section~\ref{sect:example} presents  examples of  applications of Problem~\ref{problem:1} in  two regression models.  In Section~\ref{sect:rmi}, we present the monotone inclusion reformulation of Problem~\ref{problem:1}, which will be crucial for the design of our distributed algorithm.
Subsequently, in Section~\ref{sect:PDFR}, we develop a new primal-dual forward reflected algorithm with linesearch for solving Problem~\ref{problem:mi} and analyse its convergence. This method is leveraged to obtain a distributed algorithm for Problem~\ref{problem:1} in Section~\ref{sect:distalgorithm}. Finally, conclusions and future lines of research are discussed in Section~\ref{sect:conclusion}.

\section{Notation and definitions}\label{sect:preliminaries}

Throughout this work, we use $\Hilbert$ and $\G$ to denote finite-dimensional real \emph{Hilbert spaces}. Unless otherwise stated, and to simplify the notation, we will employ $\langle \cdot, \cdot \rangle$ and $\|\cdot\|$ to denote the \emph{inner product} and the \emph{induced norm}, respectively, of any space. In most of the cases, $\Hilbert$ will just stand for the Euclidean space, which is denoted as $\R^p$ for a given dimension $p$, with norm given by the standard dot product. The set of natural numbers is $\N:=\{0,1,2,\ldots\}$. For a sequence in $\Hilbert$ we use $(z^k)$ instead of the more lengthy $(z^k)_{k\in\N}$. We use $\to$ for denoting convergence of a sequence,  $\ri D$ for the \emph{relative interior} of a set $D$, and establish $1/0=+\infty$ by convention.

When addressing Problem~\ref{problem:1}, it will be convenient to decouple the variables and work on the product space
\[
\mathcal{H}^n:=\bigl\{ \mathbf{x}=(x_1,x_2,\ldots,x_n) \, : \, x_1,x_2,\ldots,x_n \in\Hilbert \bigr\}, 
\]
endowed with the inner product 
\[
\langle \bx, \by\rangle = \langle x_1, y_1 \rangle +  \langle x_2, y_2 \rangle +\cdots + \langle x_n, y_n \rangle, \quad \bx,\by\in\Hilbert^n.
\]
We use boldface for vectors in product spaces. The \emph{diagonal subspace}  of $\Hilbert^n$ is defined as
\[
\mathcal{D}_{\Hilbert^n} := \bigl\{ \mathbf{x} \in\Hilbert^n \, :  \, x_1 = x_2 = \ldots = x_n \}.
\] 
If $\Hilbert =\R^p$, then we will identify $\Hilbert^n$ with the space $\mathbb{R}^{n\times p}$, so that $\mathbf{x}=(x_1,x_2,\ldots,x_n)^\top$ and the induced norm coincides with the Frobenius norm.
 We will often consider a matrix $V\in\mathbb{R}^{n\times n}$ as a linear operator from $\mathbb{R}^{n\times p}$ to $\mathbb{R}^{n\times p}$. Of special importance will be those matrices  that are said to \emph{enforce consensus} on the variables of $\mathbb{R}^{n\times p}$, namely, verify $\ker{V} = \mathcal{D}_{\mathbb{R}^{n\times p}}$. Equivalently,  $V\mathbf{x}=0$ if and only if the $n$ rows of $\mathbf{x}\in \mathbb{R}^{n\times p}$ coincide.  Observe that this property is independent of~$p\in\N{\setminus}\{0\}$. Indeed, one can simply say that  $V\in\mathbb{R}^{n\times n}$ enforces consensus if, when viewed as a linear operator on $\R^n$, it satisfies $\ker{V} = \mathcal{D}_{\R^n}=\operatorname{span}\{\mathbbm{1}_n\}$ with $\mathbbm{1}_n\in\R^n$ being the vector of all ones.

Let $f:\Hilbert\to{(-\infty,+\infty]}$ be a function. The \emph{domain} of $f$ is $\dom f:=\{x\in\Hilbert\,:\, f(x)<+\infty\}$. The function $f$ is said to be \emph{proper} if $\dom f \neq \emptyset$, and \emph{lower semicontinuous} if, for all $\bar{x}\in\Hilbert$,  it holds $\liminf_{x\to\bar{x}} f(x) \geq f(\bar{x})$. 

We use $A:\Hilbert\setto\Hilbert$ for denoting a \emph{set-valued operator} in $\Hilbert$, in opposition to $B:\Hilbert\to\Hilbert$, that stands for a \emph{single-valued operator}. The \emph{domain} and \emph{graph} of $A$ are given by $\dom A := \{x\in\Hilbert\,:\, A(x)\neq \emptyset\}$ and $\gra A:=\{(x,u)\in\Hilbert\times\Hilbert\,:\, u\in A(x)\}$, respectively. In particular, observe that by writing $B:\Hilbert\to\Hilbert$ we readily imply $\dom B= \Hilbert$. The closure of the domain of an operator $A$ is denoted by $\overline{\dom}\, A$.  The \emph{inverse operator} of $A$ is the operator $A^{-1}$ such that $\gra A^{-1} = \{(u,x)\in\Hilbert\times\Hilbert\,:\, u\in A(x)\}$. The \emph{identity operator} is denoted by $\Id$. Given a linear operator $K:\Hilbert\to\G$, we use $K^*:\G\to\Hilbert$ to denote its \emph{adjoint}, which is the only linear operator such that $\langle Kx,y\rangle = \langle x,K^*y\rangle$, for all $(x,y)\in\Hilbert\times\G$. We use $\|K\|$ for the  (2-)norm of~$K$.

Let $U\subseteq \Hilbert$ be a nonempty set. An operator $B:\Hilbert\to\Hilbert$ is   \emph{Lipschitz continuous} on $U$ if  there exists a constant $L\geq 0$ such that 
\[
\|B(x)-B(y)\| \leq L \|x-y\|, \quad \text{for all } x,y\in U.
\]
We say that $B$ is \emph{locally Lipschitz continuous} if it is Lipschitz continuous on some neighbourhood of every point in $\Hilbert$. 

A (possibly) set-valued operator $A:\Hilbert\setto\Hilbert$ is said to be \emph{monotone} if
\[
\langle x-y, u-v\rangle \geq 0, \quad \forall (x,u),(y,v)\in\gra A.
\]
Furthermore, if there exists no monotone operator $C:\Hilbert\setto\Hilbert$ such that $\gra C$ properly contains $\gra A$, we say that $A$ is \emph{maximally monotone}. A  well-known example of maximally monotone operator is the \emph{subdifferential} of a proper lsc convex function $f:\Hilbert\to{(-\infty,+\infty]}$, which is the set-valued operator defined as
\[
\partial f \colon   \Hilbert   \setto   \Hilbert \colon x \mapsto   \left\{ u\in\Hilbert \, : \, f(x) + \langle u, y-x \rangle \leq f(y), \; \forall y\in\Hilbert \right\}.
\]
If $f=\iota_U$ is the \emph{indicator function} of a nonempty closed convex set $U\subseteq\Hilbert$, its subdifferential is given by the \emph{normal cone} of $U$, that is,
\[
\partial_{f}(x) = N_U(x):= 
\left\{
\begin{array}{ll}
	\{ u \in \Hilbert \, : \, \sup \, \langle u, U-x \rangle \leq 0 \}, & \text{if } x\in U; \\
	\emptyset, & \text{otherwise}.
\end{array}\right.
\]

Numerical algorithms devoted to addressing problems involving  monotone operators usually make use of the so-called \emph{resolvent operator}. The resolvent  of a maximally monotone operator $A:\Hilbert\setto\Hilbert$  is the single-valued operator $J_A:\Hilbert\to\Hilbert$ given by
\[
J_A = (\Id + A)^{-1}.
\]
While resolvent operators are not always easily calculated, closed formulas are known for a wide range of useful operators; see, e.g.,~\cite[Chapter~23]{bauschke2017convex}. In particular, when $A=\partial f$ with $f$ proper lsc and convex, the resolvent reduces to the \emph{proximity operator} of $f$, namely,
\[
J_{\partial f}(x) = \prox_f(x) := \argmin_{z} \bigl\{  f(z) + \frac{1}{2} \|z-x\|^2 \bigr\}.
\]
Let $A_1:\Hilbert\setto\Hilbert$ and $A_2:\G\setto\G$ be two maximally monotone operators. The Cartesian product $A:=A_1\times A_2:\Hilbert\times\G\setto\Hilbert\times\G$, defined as $A(\mathbf{x})=(A_1(x_1),A_2(x_2))$, where $\mathbf{x}=(x_1,x_2)\in\Hilbert\times\G$,  is also a maximally monotone operator~\cite[Proposition~20.23]{bauschke2017convex}. Furthermore, its resolvent is given component-wise as $J_A=(J_{A_1},J_{A_2})$, see~\cite[Proposition 23.18]{bauschke2017convex}.

\section{Application to regression models}\label{sect:example}

The study of decentralised algorithms for Problem~\ref{problem:1} is motivated by two regression models.

\subsection{Distributionally robust Poisson regression}

The first of our example consists of a distributionally robust optimisation approach to a Poisson regression problem.

\begin{example}\label{ex:1}
	Suppose $n$ agents cooperate to estimate a parameter vector $x\in\mathbb{R}^p$ for a Poisson regression model across $q$ operational environments. More precisely, each agent $i\in\{1,2,\ldots,n\}$ collects feature vectors $a_{i,j}\in\mathbb{R}^p$ and event counts $b_{i,j}\in\mathbb{N}$ under environment $j\in\{1,2,\ldots,q\}$. Maximum-likelihood estimation leads to maximising the probability terms
	\[
	x \mapsto \frac{\bigl(\exp(\langle a_{i,j}, x \rangle)\bigr)^{b_{i,j}} \exp\bigl(-\exp(\langle a_{i,j},x\rangle)\bigr)}{b_{i,j}!}, \quad i \in\{1,2,\ldots,n\}, \, j\in\{1,2,\ldots,q\}.
	\]
	This is in turn equivalent to minimising the negative log-likelihood, which leads to the following family of convex local loss functions
	\[
	\ell_{i,j}(x) = \exp(\langle a_{i,j},x \rangle) - b_{i,j} \langle a_{i,j},x \rangle, \quad i \in\{1,2,\ldots,n\}, \, j\in\{1,2,\ldots,q\}.
	\]
	To ensure robustness against unknown environment variations, the maximum-likelihood estimation can be framed into a distributionally robust optimisation framework, leading to the problem
	\begin{equation}\label{eq:DRO}
		\min_{x\in\mathbb{R}^p} \max_{y\in \Delta^q} \; \lambda \|x\|_1 + \sum_{i=1}^n \biggl( \sum_{j=1}^q y_j \ell_{i,j}(x) - \frac{\mu}{12 n}\sum_{j=1}^q y_j^4 \biggr),
	\end{equation}
	where $\Delta^q$ is the probability simplex of dimension $q$ and the regulariser $\lambda \| \cdot\|_1$, with $\lambda >0$, aims to promote sparsity on the parameter vector. Here, the dual variable $y = (y_1, \ldots, y_q)^T \in \Delta^q$ represents an adversarial probability distribution over the $q$ environmental conditions. Finally, the power penalty term $-\frac{\mu}{12n}\sum_{j=1}^q y_j^4$ derives from the Cressie--Read family of divergences~\cite{MR790631} and is used to prevent adversarial collapse onto a single worst-case environment; see, e.g.,~\cite{DuchiDivergences,MR4379565}.
	
	It is easy to see that~\eqref{eq:DRO} is in the form of~\eqref{eq:min-max} by setting $f_i(x):=\frac{\lambda}{n} \|x \|_1$, $\phi_i(x,y) := \sum_{j=1}^q \bigl(y_j \ell_{i,j}(x) - \frac{\mu}{12 n} y_j^4 \bigr)$, and $g_i(y) := \frac{1}{n}\iota_{\Delta^q}(y)$. Observe that 
    $\phi_i(\cdot,y)$ is convex for all $y\in [0,+\infty)^q$, $\phi_i(x,\cdot)$ is concave for all $x \in \R^p$, $\phi_i$ is smooth on the entire space $\mathbb{R}^p \times \mathbb{R}^q$, and both partial gradients $\nabla_x \phi_i(\cdot,\cdot)$ and $\nabla_y\phi_i(\cdot,\cdot)$ are locally Lipschitz continuous, but not globally.
\end{example}

Further examples illustrating the scope of Problem~\ref{problem:1} can be obtained by considering different loss functions. For instance, one can frame the multiple image deconvolution problem from corrupted Poisson data (see, e.g.,~\cite{atenas2024distributed,refId0}) within a distributionally robust setting.

\subsection{Distributed regression with high-order powers}

The second example involves linear regression for data contaminated with light-tailed noise. Consider the linear model $b = \langle a, x\rangle + \varepsilon$, where the noise vector $\varepsilon$ distributes according to a density known a priori to have light tails. In this setting, measuring the misfit with power of higher order may be statistically more efficient than least-squares, as the former can control larger residuals more aggressively.

\begin{example}\label{ex:2}
Let $n$ agents in a network cooperate to estimate a sparse parameter vector $x\in\R^p$ of a linear model. Each agent $i\in\{1,2,\ldots,n\}$ privately holds a design matrix $A_i\in\R^{m_i\times p}$, with rows $a_{i,1}^\top,\ldots,a_{i,m_i}^\top$, and a vector of observations $b_i\in\R^{m_i}$ whose noise is known to be light-tailed. In addition, the estimate is required to satisfy a system of $q$ linear inequality constraints encoding prior knowledge on the model, whose data are split among the agents. Given $\lambda>0$, the optimisation problem can thus be formulated as \begin{equation}\label{eq:L4primal}
		\min_{x\in\R^p} \; \lambda\|x\|_1 + \sum_{i=1}^n \ell_i(x) \quad \text{s.t.} \quad \sum_{i=1}^n (C_i x - d_i) \leq 0, \end{equation} 
        where, for all $i\in\{1,2,\ldots,n\}$, $C_i\in\R^{q\times p}$, $d_i\in\R^q$ and
        \begin{equation*}
		    \ell_i(x) := \frac{1}{4}\|A_i x - b_i\|_4^4 = \frac{1}{4}\sum_{j=1}^{m_i} \bigl(\langle a_{i,j},x\rangle - b_{i,j}\bigr)^4.
		\end{equation*} Dualising the coupling constraint with a multiplier $y\in\R^q_+$ yields the Lagrangian formulation \begin{equation}\label{eq:L4minmax}
		\min_{x\in\R^p}\max_{y\in\R^q_+} \; \lambda\|x\|_1 + \sum_{i=1}^n \Bigl( \ell_i(x) + \bigl\langle y, C_i x - d_i\bigr\rangle \Bigr).
	\end{equation} Problem~\eqref{eq:L4minmax} is in the form of~\eqref{eq:min-max} by setting $f_i(x):=\frac{\lambda}{n}\|x\|_1$, $g_i(y):=\frac{1}{n}\iota_{\R^q_+}(y)$ and $\phi_i(x,y):=\ell_i(x)+\langle y, C_i x-d_i\rangle$. Note that $\phi_i$ is convex-concave, and the partial gradients $\nabla_x \phi_i(\cdot,\cdot)$ and $\nabla_y\phi_i(\cdot,\cdot)$ are locally Lipschitz continuous.
\end{example}

\section{Reformulation as a monotone inclusion}\label{sect:rmi}

Under the assumption that Problem~\ref{problem:1} has at least a solution, we now obtain a reformulation as a monotone inclusion. First, observe that (\ref{eq:min-max}) can be alternatively written as 
\begin{equation}\label{eq:minimax_cons}
\min_{(x_1,\ldots,x_n)\in\R^{p\times n}} \max_{(y_1,\ldots,y_n)\in\R^{q\times n}} \sum_{i=1}^n \bigl(f_i(x_i) + \phi_i(x_i,y_i) - g_i(y_i) \bigr) \quad \text{s.t.} \quad
\left\{
\begin{array}{l}
x_1 = x_2 = \ldots = x_n, \\
y_1 = y_2 = \ldots = y_n.
\end{array}
\right.
\end{equation}
For convenience, let $f:\R^{n\times p}\to\Rext$, $g:\R^{n\times q}\to\Rext$ and $\phi:\R^{n\times p}\times\R^{n\times q}\to\R$ denote the functions
\begin{equation}\label{eq:fgphi}
f(\bx) = \sum_{i=1}^n f_i(x_i), \quad g(\by) = \sum_{i=1}^n g_i(y_i)  \quad \text{and} \quad \phi(\bx,\by) = \sum_{i=1}^n \phi_i(x_i,y_i),
\end{equation}
where $\bx := (x_1,\ldots,x_n)^{\top}\in\R^{n\times p}$  and $\by:=(y_1,\ldots,y_n)^{\top}\in\R^{n\times q}$.
Let $K_1, K_2\in\R^{n\times n}$ be matrices ensuring consensus on $\R^{n\times p}$ and $\R^{n\times q}$, respectively.  Then (\ref{eq:minimax_cons}) is equivalent to the problem
\begin{equation}\label{eq:minimax_pd0}
\min_{\bx\in\R^{n\times p}}\max_{\by\in\R^{n\times q}} f(\bx) + \iota_{\{0\}}(K_1\bx) + \phi(\bx,\by) - \iota_{\{0\}}(K_2\by) - g(\by).
\end{equation}

We will tackle~(\ref{eq:minimax_pd0}) by addressing the monotone inclusion resulting from its first order optimality conditions. 
Let $\bz:=(\bx,\by)$. Let the operators $A: \R^{n\times p}\times\R^{n\times q}\setto \R^{n\times p}\times\R^{n\times q}$, $B: \R^{n\times p}\times\R^{n\times q}\to \R^{n\times p}\times\R^{n\times q}$ and the linear operator $K:\R^{n\times p}\times\R^{n\times q}\to \R^{n\times p}\times\R^{n\times q}$ be
\begin{equation}\label{eq:opreform} 
A(\bz) := \bigl(\partial f(\bx), \partial g(\by)\bigr), \quad B(\bz):=\bigl(\nabla_{\bx} \phi(\bx,\by), -\nabla_{\by} \phi(\bx,\by)\bigr) \quad \text{and} \quad K\bz:= (K_1\bx, K_2\by),
\end{equation}
where $\partial f = \bigtimes_{i=1}^n \partial f_i$ and $\partial g = \bigtimes_{i=1}^n \partial g_i$ by~\cite[Proposition~16.9]{bauschke2017convex}.
In particular, the operator $A$ is maximally monotone as the Cartesian product of maximally monotone operators, while $B$  is locally Lipschitz continuous on $\overline{\dom}\, A$ and monotone on $\dom A$ by assumption.
Hence, by the subdifferential inclusion~\cite[Proposition~16.6]{bauschke2017convex}, a sufficient condition for a  vector $\bz\in \R^{n\times p}\times\R^{n\times q}$ to be a solution of~(\ref{eq:minimax_pd0}) is to solve the monotone inclusion problem consisting in 
 \begin{equation}\label{eq:mm_as_mi}
 \text{find } \bz\in \R^{n\times p}\times\R^{n\times q} \text{ such that } 0\in \bigl(A+B\bigr)(\bz) + K^*N_{\{0\}}\bigl(K\bz\bigr).
 \end{equation}
 Clearly, the above is in the form of Problem~\ref{problem:mi}. If in addition we  want to ensure that both~(\ref{eq:min-max}) and~(\ref{eq:mm_as_mi}) are equivalent, we further need the following subdifferential sum rules to hold:
 \begin{equation}\label{eq:sumrule}
 	\partial f + K_1^* \circ \partial \iota_{\{0\}} \circ K_1 =  \partial \left(f + \iota_{\{0\}} \circ K_1 \right), \quad \partial g + K_2^* \circ \partial \iota_{\{0\}} \circ K_2 =  \partial \left(g + \iota_{\{0\}} \circ K_2 \right).
 \end{equation}
 By~\cite[Theorem~16.47 \& Proposition~6.19(x)]{bauschke2017convex} for~\eqref{eq:sumrule} to hold it suffices that $\cap_{i=1}^n \ri(\dom f_i) \neq \emptyset$ and $\cap_{i=1}^n \ri(\dom g_i) \neq \emptyset$. Observe that the latter trivially hold in Examples~\ref{ex:1} and~\ref{ex:2}. We summarise all of the above in the following proposition.
 
 \begin{proposition}\label{prop:sp-mi}
 	Suppose that Problem~\ref{problem:1} has at least a solution and that the subdifferential sum rules in~\eqref{eq:sumrule} hold. Then a vector $\bz^*=(\bx^*,\by^*)\in \R^{n\times p}\times\R^{n\times q}$ solves~\eqref{eq:mm_as_mi} if and only if $x^*:=x_1^*=\ldots=x_n^*\in\R^p$, $y^*:=y_1^*=\ldots=y_n^*\in\R^q$ and $(x^*,y^*)$ is a saddle point of~\eqref{eq:min-max}.
 \end{proposition}

Proposition~\ref{prop:sp-mi} motivates the development of numerical schemes for solving Problem~\ref{problem:mi}.

\section{A new primal-dual forward-reflected splitting algorithm with linesearch}\label{sect:PDFR}

In this section, we propose a new splitting method with linesearch for addressing Problem~\ref{problem:mi}. Hence, in what follows we assume that $\Hilbert$ and $\G$ are finite-dimensional Hilbert spaces.

The presence of the linear compositions over the operator $C$ makes it convenient to address the resolution of~(\ref{eq:monotone_inclusion}) by means of its associated \emph{primal-dual problem}. That is, we search for a pair $(z^*,u^*)\in\Hilbert\times\G$ such that 
\begin{equation}\label{eq:pdincl}
\left[
\begin{array}{c}
0 \\
0
\end{array}
\right] \in \left[
\begin{array}{cc}
C^{-1} &  -K \\
K^* & A+B
\end{array}
\right]
\left[
\begin{array}{c}
u^* \\
z^*
\end{array}
\right].
\end{equation}
We refer to such a pair as a \emph{primal-dual solution} of~(\ref{eq:monotone_inclusion}). 
It is a simple algebraic exercise to check that a pair $(z^*,u^*)$ is a primal-dual solution if and only if $z^*$ is a solution of the (primal) problem~(\ref{eq:monotone_inclusion}) and $u^*$ provides a solution of an \emph{Attouch--Th\'era dual} problem~\cite{MR1422748}.

In order to obtain a primal-dual solution of~(\ref{eq:monotone_inclusion}), we develop a splitting algorithm that incorporates a linesearch procedure to account for the lack of knowledge of the (local) Lipschitz modulus of the operator $B$. Our method does assume knowledge of the norm of the linear operator $K$. We name this method \emph{primal-dual forward-reflected with linesearch} (PDFRL in short)   and present its pseudocode in Algorithm~\ref{alg:pdls}. Observe that the resolvent in the dual update uses the stepsize computed in the previous iteration of the algorithm. In principle, the order of Step~1 and Step~2 can be interchanged, however this would require computing an extra resolvent of $C^{-1}$ for every inner iteration of the linesearch, increasing the computational cost of the algorithm.

\begin{algorithm}
\caption{Primal-dual forward-reflected with  linesearch (PDFRL) for Problem~\ref{problem:mi}}\label{alg:pdls}
Choose initial points $z^{-1}=z^{0}\in \dom A$, $u^0\in\G$\;
Choose  parameters $\delta\in{(0,1)}$ and  $\beta\in{(0,1)}$\;
Choose an initial stepsize $\tau_{-1} \in{\left(0, \delta \frac{1-\beta}{2\|K\|}\right)}$\;
\For{$k=0,1,\ldots$}
{

$\rhd$ \textbf{Step 1} (Dual update). Define
\begin{equation}\label{eq:dual_update}
u^{k+1} = J_{\tau_{k-1} C^{-1}} \big( u^k + \tau_{k-1} K z^{k} \big);
\end{equation}

$\rhd$ \textbf{Step 2} (Abstract linesearch and primal update). Compute
\begin{equation}\label{eq:primal_update}
z^{k+1} = J_{\tau_k A}\big( z^k - \tau_k K^*u^{k+1} - \tau_{k-1} K^*(u^{k+1}-u^k) - \tau_k B(z^k) - \tau_{k-1} (B(z^k)-B(z^{k-1}) ) \big),
\end{equation}
where $\tau_k \in{\left(0,\delta\frac{1-\beta}{2\|K\|}\right)}$ satisfies
\begin{equation}\label{eq:linesearch}
\tau_k \|B(z^{k+1})-B(z^k)\| \leq \frac{\sqrt{\beta}}{2} \|z^{k+1}-z^k\|.
\end{equation}
}
\end{algorithm}

\begin{remark}\label{remark:particularcases}
We collect some algorithms that can be recovered as special cases or are related to Algorithm~\ref{alg:pdls}.
\begin{enumerate}[label=(\roman*)]

\item\label{remark:pc-1} If $B$ is globally Lipschitz, a fixed stepsize $\tau$ can be taken according to the Lipschitz modulus. The update of the algorithm will then read as
\begin{equation}\label{eq:pdtr}
\left\{
\begin{aligned}
	u^{k+1} &= J_{\tau C^{-1}} \big( u^k + \tau K z^{k} \big), \\
	z^{k+1} &= J_{\tau A}\big( z^k - \tau K^* (2 u^{k+1} -u^k) - 2 \tau B(z^k) + \tau B(z^{k-1})  \big).
\end{aligned}
\right.
\end{equation}
This scheme is  the \emph{primal-dual twice reflected} algorithm in~\cite{malitsky2023first}, but with extrapolation performed on  the dual variable. Indeed, analogously to the algorithm in~\cite{malitsky2023first},  iteration~\eqref{eq:pdtr} can be recovered as an instance of the \emph{forward-reflected-backward method} (FoRB)~\cite{MR4101375} under a different metric. Namely, \eqref{eq:pdtr} is given by the scheme
\[
\omega^{k+1} = J_{M^{-1} G} \bigl(\omega^k - 2 M^{-1} F(\omega^k)  + M^{-1} F(\omega^{k-1})\bigr),
\] 
with $\omega^k = (z^k,u^k)$ and where the matrix $M$ and the operators $G, F$  are given by
\[
M = 
\begin{bmatrix}
	\frac{1}{\tau} \Id & K^* \\ K & \frac{1}{\tau} \Id 
\end{bmatrix},
\quad 
G = \begin{bmatrix}
	A &  K^* \\ -K & C^{-1} 
\end{bmatrix}
\quad 
\text{and}
\quad
F = \begin{bmatrix}
	B & 0 \\ 0 & 0
\end{bmatrix}.
\]

\item If in addition $B=0$, we recover the scheme
\[
\left\{
\begin{aligned}
	u^{k+1} &= J_{\tau C^{-1}} \big( u^k + \tau K z^{k} \big), \\
	z^{k+1} &= J_{\tau A}\big( z^k- \tau K^*(2u^{k+1}-u^k)   \big),
\end{aligned}
\right.
\]
which corresponds to an instance of the \emph{primal-dual hybrid gradient} method~\cite{MR2782122}.
\item If $K=0$, then updates~\eqref{eq:dual_update}-\eqref{eq:primal_update} become
\begin{equation}\label{eq:pc-3}
\left\{
\begin{aligned}
u^{k+1} &= J_{\tau_{k-1}C^{-1}} \big(u^k\big), \\
z^{k+1} &= J_{\tau_k A}  \bigl(z^k - \tau_k B(z^k) - \tau_{k-1}(B(z^k)-B(z^{k-1}))\bigr).
\end{aligned}
\right.
\end{equation}
Hence, two independent iterations are performed by the method: a variable stepsize \emph{proximal point algorithm} for $C^{-1}$ and the \emph{forward-reflected-backward with linesearch}~\cite[Algorithm~3.1]{MR4101375} for  $A+B$.


\item If in addition $A=0$ and $B$ is globally Lipschitz continuous, the $(z^k)$ iteration in~\eqref{eq:pc-3} reduces to
\[
z^{k +1} = z^k  -  2\tau B(z^k) + \tau B(z^{k-1}).
\]
This is the so-called \emph{optimistic gradient descent ascent} (OGDA) (also known as \emph{optimistic mirror descent}) within the GAN literature~\cite{daskalakis2018training}, which in turn is recovered from the unconstrained version of \emph{Popov's algorithm}~\cite{popov1980modification} 
\[
\left\{
\begin{aligned}
	\bar{z}^{k+1} &= \bar{z}^k - \tau B(z^k), \\
	 z^{k+1}  &= \bar{z}^{k+1} - \tau B(z^k),
\end{aligned}
\right.
\]
under the initialisation $\bar{z}^0 = z^0 + \tau B(z^{-1})$.
\end{enumerate}
\end{remark}

\subsection{Convergence analysis}
In this section, the convergence of Algorithm~\ref{alg:pdls} is established.
First, the following lemma ensures that the linesearch in Step~2  is well-defined. 
\begin{lemma}\label{lemma:lswd}
Suppose that the assumptions of Problem~\ref{problem:mi} hold. Consider the sequences $(z^k)$ and $(u^k)$ generated by Algorithm~\ref{alg:pdls}. Then,  for all $k\in\N$, there exists $\zeta_k>0$ such that  
\[
\tau \|B(z^{k+1})-B(z^k)\| \leq \frac{\sqrt{\beta}}{2}\|z^{k+1}-z^k\|, \quad \text{for all } \tau\in{(0,\zeta_k)}.
\]
\end{lemma}
\begin{proof}
Let $k\in \N$ fixed. For convenience let us denote 
\begin{equation}\label{eq:zk1tau}
z^{k+1}(\tau):= J_{\tau A}\big( z^k - \tau K^*u^{k+1} - \tau_{k-1} K^*(u^{k+1}-u^k) - \tau B(z^k) - \tau_{k-1} (B(z^k)-B(z^{k-1}) ) \big),
\end{equation}
 with $\tau$ taking values in the nonnegative real line. In particular, we have 
\begin{equation}\label{eq:zk10}
z^{k+1}(0) = z^k - \tau_{k-1}K^*(u^{k+1}-u^k)- \tau_{k-1}(B(z^k)-B(z^{k-1})).
\end{equation}
By \cite[Proposition~3.4]{MR5098995}, $z^{k+1}(\tau)\to P_{\overline{\dom}\, A}\big(z^k - \tau_{k-1} K^*(u^{k+1}-u^k) - \tau_{k-1} (B(z^k)-B(z^{k-1}) )\bigr) = P_{\overline{\dom}\, A}\big(z^{k+1}(0)\bigr) $ as $\tau \,{\downarrow}\, 0$, where $ P_{\overline{\dom}\, A}$ stands for the projection onto the closure of the domain of~$A$.

Let $U\subseteq\overline{\dom}\, A$ be a bounded convex set such that it is open with respect to $\overline{\dom}\, A$ and that contains the pair $\bigl\{z^{k}, P_{\overline{\dom}\, A}\big(z^{k+1}(0)\bigr)\bigr\}$. Since we have proved that  $z^{k+1}(\tau)\subset\dom A$ converges to $P_{\overline{\dom}\, A}\big(z^{k+1}(0)\bigr)$ as  $\tau \,{\downarrow}\, 0$, then for all sufficiently small $\tau$ we can ensure that $z^{k+1}(\tau)$ belongs to~$U$.  Recall now that as a locally Lipschitz continuous operator on a finite dimensional space, $B$ is Lipschitz continuous on any   nonempty compact convex subset, and in particular on  $\overline{U}\subseteq\overline{\dom}\, A $. Hence, there exist constants $L\geq 0$  and $\eta_k>0$  such that
\[
\tau \bigl\| B\bigl(z^{k+1}(\tau)\bigr) - B\bigl(z^k\bigr) \big\| \leq \tau L \bigl\| z^{k+1}(\tau) - z^k\|, \quad \text{for all }   \tau\in{(0,\eta_k)}.
\]
To conclude the claim it suffices to take $\zeta_k:=\min\{\eta_k, \sqrt{\beta}/{(2L)}\}$.

\end{proof}

As mentioned in Remark~\ref{remark:particularcases}\ref{remark:pc-1}, Algorithm~\ref{alg:pdls} is closely related to the \emph{primal-dual twice reflected} (PDTR) method in~\cite{malitsky2023first}. While the convergence analysis of PDTR draws upon its identification as an instance of the forward-reflected-backward algorithm~\cite{MR4101375} under a matrix-induced metric, we here make use of a direct Lyapunov-based analysis to establish the convergence of Algorithm~\ref{alg:pdls}. 
Fix a pair $(z,u)\in\Hilbert\times\G$, and for all $k\in\N$, we define the energy functions
\begin{align}
&\cE^{(z,u)}_k := \frac{1}{2}\big( \|z^k-z\|^2 + \|u^k-u\|^2\big) + \tau_{k-1} \langle K(z^k-z), u^k-u \rangle, \label{eq:energy} \\
&\cV_k^{(z,u)} := \cE_k^{(z,u)} - \tau_{k-1}\langle B(z^k)-B(z^{k-1}), z^k-z\rangle + \frac{1}{4}\|z^k-z^{k-1}\|^2. \label{eq:Vzu}
\end{align}
Note that, by  Cauchy--Schwarz and Young's inequalities, $\cE^{(z,u)}_k \geq 0$ whenever $0 < \tau_{k-1}\|K\| \leq 1$. 
In addition,   $\cV_k^{(z,u)}$ is also nonnegative thanks to the linesearch condition in Algorithm~\ref{alg:pdls}. We prove this in the following result.
\begin{lemma}\label{lemma:cVnonnegative}
  Suppose that the assumptions of Problem~\ref{problem:mi} hold.  Consider the sequences $(z^k)$ and $(u^k)$ generated by Algorithm~\ref{alg:pdls}. Then, for any fixed $(z,u)\in\Hilbert\times\G$ and for all $k\in\N$, it holds
\begin{equation}\label{eq:Vlower}
\begin{aligned}
    \cV^{(z,u)}_{k} &\geq \cE^{(z,u)}_{k} - \frac{\beta}{4}\|z^{k}-z\|^2 \\ &\geq \frac{1}{2}\left(1- \frac{\beta}{2} - \tau_{k-1}\|K\|\right) \|z^{k}-z\|^2 + \frac{1}{2} \left(1- \tau_{k-1}\|K\|\right) \|u^{k}-u\|^2 \geq 0.
\end{aligned}
\end{equation}
\end{lemma}
\begin{proof}
For the first inequality, observe that $z^{0}=z^{-1}$, by initialisation of Algorithm~\ref{alg:pdls}. This implies $\cV_0^{(z,u)}=\cE_0^{(z,u)}$, and the inequality trivially holds for $k=0$. For $k\in\N{\setminus}\{0\}$, using Cauchy--Schwarz and~\eqref{eq:linesearch} at iteration $k-1$ followed by Young's inequality, we get 
\begin{equation}\label{eq:lsestimation}\begin{aligned}
    - \tau_{k-1} \langle B(z^{k}) - B(z^{k-1}),z^{k}-z\rangle  & \geq  -  \frac{\sqrt{\beta}}{2} \|z^{k}-z^{k-1}\|\|z^{k}-z\| \\& \geq  - \frac{1}{4} \|z^{k}-z^{k-1}\|^2 - \frac{\beta}{4} \|z^{k}-z\|^2,
\end{aligned}
\end{equation} and hence $\cV^{(z,u)}_{k} \geq \cE^{(z,u)}_{k} - \frac{\beta}{4}\|z^{k}-z\|^2$. Bounding the bilinear term in~\eqref{eq:energy} by Cauchy--Schwarz and Young's inequalities yields the second inequality in~\eqref{eq:Vlower}. The third one follows from the fact that $\tau_{k-1}\|K\| < \delta(1-\beta)/2 < 1-\beta/2$, for all $k\in\N$.
\end{proof}

To conduct our convergence analysis, we first prove a technical lemma on the functions $\cV^{(z,u)}_k$ that involves the following sequences:
{\small
\begin{align}
&\omega^{k+1} := \frac{1}{\tau_{k-1}}\bigl( u^k-u^{k+1}+\tau_{k-1}Kz^k \bigr), \label{eq:c-graph_elements}\\
&\nu^{k+1} := \frac{1}{\tau_k}\bigl( z^k-z^{k+1} - \tau_k K^*u^{k+1} - \tau_{k-1}K^*(u^{k+1}-u^k)  -\tau_k B(z^k) - \tau_{k-1}(B(z^k)-B(z^{k-1})) \bigr).  \label{eq:a-graph_elements}
\end{align}}%
Note that~\eqref{eq:dual_update} and~\eqref{eq:primal_update} yield $\omega^{k+1}\in\gra C^{-1}(u^{k+1})$ and $\nu^{k+1}\in\gra{A}(z^{k+1}
)$, for all $k\in\N$.

\begin{lemma}\label{l:intermediate}
    Suppose that the assumptions of Problem~\ref{problem:mi} hold.  Consider the sequences $(z^k)$ and $(u^k)$ generated by Algorithm~\ref{alg:pdls}. For any fixed pair $(z,u)\in\Hilbert\times\G$ and all $k\in\N$, define the residual \begin{equation}\label{eq:residual}
\cR_k^{(z,u)} := \tau_k \bigl\langle z^{k+1}-z, \, \nu^{k+1} + K^*u + B(z^{k+1}) \bigr\rangle + \tau_{k-1} \bigl\langle u^{k+1}-u, \, \omega^{k+1} - Kz \bigr\rangle.
\end{equation} Then 
\begin{equation} \label{eq:intermediate}\begin{aligned}
    \cR_k^{(z,u)} + \cV^{(z,u)}_{k+1} & \leq \cV^{(z,u)}_{k} - \frac{1}{2}\left( \frac{1-\beta}{2} -  \tau_{k-1}\|K\|\right) \|z^{k+1}-z^k\|^2  & \\
    &\quad- \frac{1}{2}(1-\tau_{k-1} \|K\|) \|u^{k+1}-u^k\|^2.
\end{aligned}
\end{equation}  
\end{lemma}

\begin{proof}
    From~\eqref{eq:a-graph_elements}, we have
    \begin{equation}\label{eq:Rk-1}
    \begin{aligned}
   \tau_k \bigl\langle z^{k+1}&-z, \nu^{k+1}+K^*u+B(z^{k+1})\bigr\rangle \\
    =&\,\langle z^{k+1}-z, z^k- z^{k+1} \rangle - \tau_k \langle z^{k+1}-z, K^* (u^{k+1}-u) \rangle  - \tau_{k-1} \langle  z^{k+1}-z,K^*( u^{k+1}- u^k)\rangle \\
&\quad  - \tau_k \langle z^{k+1}-z, B(z^k)-B(z^{k+1})\rangle - \tau_{k-1} \langle z^{k+1}-z, B(z^k)- B(z^{k-1}) \rangle,
\end{aligned}
\end{equation}
where the first term in the second line above can be rearranged as
\begin{equation*}
\langle z^{k+1}-z, z^k- z^{k+1} \rangle = \frac{1}{2}\bigl( \|z^k-z\|^2 - \|z^{k+1}-z\|^2  - \|z^{k+1}-z^k\|^2\bigr).
\end{equation*} 
Similarly, from~\eqref{eq:c-graph_elements}, 
\begin{equation}\label{eq:Rk-2}
\begin{aligned}
        \tau_{k-1}\bigl\langle &u^{k+1}-u, \omega^{k+1}-Kz\bigr\rangle \\& =  \,\langle  u^{k+1}-u, u^k-u^{k+1}+\tau_{k-1}K z^k  - \tau_{k-1}Kz\rangle \\
& = \frac{1}{2} \big( \|u^k-u\|^2 - \|u^{k+1}-u\|^2  - \|u^{k+1}-u^k\|^2 \big) + \tau_{k-1} \langle u^{k+1}-u, K(z^k -z) \rangle. \end{aligned}
\end{equation}
Adding~\eqref{eq:Rk-1}-\eqref{eq:Rk-2} and rearranging terms, we get
\begin{equation}\label{eq:cE1}\begin{aligned}
   \cR_k^{(z,u)}
   + \cE^{(z,u)}_{k+1} & =  \frac{1}{2} \bigl( \|z^{k}  -z\|^2 + \|u^k-u\|^2 \bigr) -  \frac{1}{2} \bigl( \|z^{k+1}-z^k\|^2 + \|u^{k+1}-u^k \|^2 \bigr) \\
& \quad +\tau_k \langle B(z^{k+1})-B(z^k), z^{k+1}-z\rangle - \tau_{k-1}\langle B(z^k)-B(z^{k-1}), z^{k+1}-z\rangle \\
& \quad + \tau_{k-1} \langle K(z^k-z), u^{k+1}-u \rangle - \tau_{k-1} \langle K(z^{k+1}-z), u^{k+1}-u^k \rangle.
\end{aligned}\end{equation}
The last two terms in~\eqref{eq:cE1} can be rearranged as
\begin{equation}\label{eq:lterms-aux}
\begin{aligned}
\tau_{k-1} & \langle K(z^k-z), u^{k+1}-u \rangle - \tau_{k-1} \langle K(z^{k+1}-z), u^{k+1}-u^k \rangle \\
& = \tau_{k-1} \langle K(z^k-z), u^k-u \rangle + \tau_{k-1} \langle K(z^k-z), u^{k+1}-u^k \rangle    - \tau_{k-1} \langle K(z^{k+1}-z), u^{k+1}-u^k \rangle  \\ 
& = \tau_{k-1} \langle K(z^k-z), u^k-u \rangle - \tau_{k-1} \langle K(z^{k+1}-z^k), u^{k+1}-u^k \rangle \\
& \leq \tau_{k-1} \langle K(z^k-z), u^k-u \rangle +  \frac{\tau_{k-1} }{2}\|K\|\bigl( \|z^{k+1}-z^k\|^2 + \|u^{k+1}-u^k\|^2\bigr),
\end{aligned}\end{equation}
where the last inequality is due to Cauchy--Schwarz and Young's inequalities. We now draw our attention to the second term in the second line of~\eqref{eq:cE1}. By adding  and subtracting $\tau_{k-1} \langle B(z^k) -B(z^{k-1}),z^{k} \rangle$, we get
\begin{equation}\label{eq:lsestimation-aux}
\begin{aligned}
 - \tau_{k-1} & \langle B(z^k) - B(z^{k-1}), z^{k+1}-z \rangle  \\
 = & - \tau_{k-1} \langle B(z^k) - B(z^{k-1}), z^{k}-z \rangle - \tau_{k-1} \langle B(z^k) - B(z^{k-1}), z^{k+1}-z^k \rangle \\
 \leq & - \tau_{k-1} \langle B(z^k) - B(z^{k-1}), z^{k}-z \rangle + \frac{1}{4} \bigl( \|z^k -z^{k-1}\|^2 + \beta\|z^{k+1}-z^k\|^2\bigr),
 \end{aligned}
\end{equation}
where the inequality is due to~\eqref{eq:linesearch} and follows similarly to \eqref{eq:lsestimation}. 
Combining equations~\eqref{eq:lterms-aux}-\eqref{eq:lsestimation-aux} and replacing in~(\ref{eq:cE1}),  we obtain
\[
\begin{aligned}
\cR^{(z,u)}_{k}+\cE^{(z,u)}_{k+1} - \tau_k&  \langle  B(z^{k+1}) -B(z^k),z^{k+1}-z \rangle  \\
& \leq   \frac{1}{2} \bigl( \|z^{k}  -z\|^2 + \|u^k-u\|^2 \bigr)   + \tau_{k-1} \langle K(z^k-z), u^k-u \rangle   + \frac{1}{4} \|z^k -z^{k-1}\|^2  \\
& \quad - \tau_{k-1} \langle B(z^k)-B(z^{k-1}),z^{k}-z \rangle - \frac{1}{2} (1- \tau_{k-1}\|K\|) \|u^{k+1}-u^k\|^2 \\
& \quad - \frac{1}{2} \left(1 - \frac{\beta}{2} - \tau_{k-1}\|K\| \right) \|z^{k+1}-z^k\|^2.
\end{aligned}
\]
Finally, adding $\|z^{k+1}-z^k\|^2/4$ to both sides of the equation above we obtain~\eqref{eq:intermediate} by noting that the first three terms in the right-hand side define $\cV^{(z,u)}_k$.
\end{proof}

\begin{proposition}\label{prop:descent_inequality}
Suppose that the assumptions of Problem~\ref{problem:mi} hold.  Let $(z^*,u^*)\in\Hilbert\times\G$ be a solution of the primal-dual inclusion~(\ref{eq:pdincl}). Let   $(z^k)$ and $(u^k)$ be the sequences generated by Algorithm~\ref{alg:pdls}. The following assertions hold:
\begin{enumerate}[label=(\roman*)]
\item\label{it:di-1} for all $k\in\N$, it holds
\begin{equation}\label{eq:cE}
\begin{aligned}
\cV^{(z^*,u^*)}_{k+1} &\leq \cV^{(z^*,u^*)}_k  - \frac{1}{2}\left( \frac{1-\beta}{2} -  \tau_{k-1}\|K\|\right) \|z^{k+1}-z^k\|^2 \\
& \quad- \frac{1}{2}(1-\tau_{k-1} \|K\|) \|u^{k+1}-u^k\|^2;
\end{aligned}
\end{equation}
\item\label{it:di-2} the sequences $(\|z^{k+1}-z^k\|)$ and  $(\|u^{k+1}-u^k\|)$ converge to $0$;
\item\label{it:di-3} the sequences $(z^k)$ and $(u^k)$ are bounded.
\end{enumerate}
\end{proposition}
\begin{proof}
\ref{it:di-1} By Lemma~\ref{l:intermediate} it suffices to show that the residual~\eqref{eq:residual} is nonnegative at $(z,u)=(z^*,u^*)$. Indeed, from \eqref{eq:pdincl}, $c^*:=Kz^*\in C^{-1}(u^*)$ and $a^*:=-K^*u^*-B(z^*)\in A(z^*)$, and thus
{\small
\[
\begin{aligned}
\cR_k^{(z^*,u^*)} & = \tau_k \bigl\langle z^{k+1}-z^*, \, \nu^{k+1} - a^* + B(z^{k+1}) - B(z^*) \bigr\rangle + \tau_{k-1} \bigl\langle u^{k+1}-u^*, \, \omega^{k+1} - c^* \bigr\rangle \\
& = \tau_k \bigl\langle z^{k+1}-z^*, \nu^{k+1} - a^*\bigr\rangle + \tau_k \bigl\langle z^{k+1}-z^*, B(z^{k+1})-B(z^*)\bigr\rangle + \tau_{k-1} \bigl\langle u^{k+1}-u^*, \omega^{k+1}-c^*\bigr\rangle \\
& \geq 0,
\end{aligned}
\]}%
where the inequality follows from the monotonicity of $A$, of $B$, and of $C^{-1}$, by recalling that $\omega^{k+1}\in\gra C^{-1}(u^{k+1})$ and $\nu^{k+1}\in\gra{A}(z^{k+1}
)$, for all $k\in\N$.

\ref{it:di-2} Let $N\in\N{\setminus}\{0\}$. By summing~(\ref{eq:cE}) from $k=0$ to $N$, telescoping and rearranging we get
\begin{equation}\label{eq:tel-1}
\begin{aligned}
 \cV^{(z^*,u^*)}_{0} & \geq \cV^{(z^*,u^*)}_{0}- \cV^{(z^*,u^*)}_{N+1}  \\ 
 & \geq \frac{1}{2}\sum_{k=0}^N \left(\frac{1-\beta}{2} -\tau_{k-1}\|K\|\right) \|z^{k+1}-z^k\|^2 +\frac{1}{2}\sum_{k=0}^N   (1- \tau_{k-1}\|K\|) \|u^{k+1}-u^k\|^2,
\end{aligned}
\end{equation}
where the first inequality is due to Lemma~\ref{lemma:cVnonnegative}. Hence, the sums in the right-hand side of~(\ref{eq:tel-1}) are finite when $N\to\infty$. Furthermore, since $1-\tau_{k-1}\|K\| > (1-\beta)/2 - \tau_{k-1}\|K\| > (1-\delta) (1-\beta)/2 >0$, for all $k\in\N$,  we conclude that $(\|z^{k+1}-z^k\|)$ and $(\|u^{k+1}-u^k\|)$ converge to $0$ as $k\to\infty$.

\ref{it:di-3} By combining~\eqref{eq:tel-1} and~\eqref{eq:Vlower}, there exists a constant $R>0$ such that, for all $N\in\N{\setminus}\{0\}$,
\begin{equation}\label{eq:bound-3}
\begin{aligned}
R &\geq \cV_{N+1}^{(z^*,u^*)} \\
& \geq
\frac{1}{2}\left(1- \frac{\beta}{2} - \tau_{N}\|K\|\right) \|z^{N+1}-z^*\|^2 + \frac{1}{2} \left(1- \tau_{N}\|K\|\right) \|u^{N+1}-u^*\|^2.
\end{aligned}
\end{equation}
Since $\tau_{N}<  \delta (1-\beta)/(2\|K\|)$, we have  $1-\tau_{N}\|K\|  > 1-\beta/2- \tau_{N}\|K\| > 1/2 + (1-\delta)(1-\beta)/2 >1/2$. This leads to the bound
\begin{equation}\label{eq:bound-4}
\frac{1}{4}\|z^{N+1}-z^*\|^2 + \frac{1}{4}\|u^{N+1}-u^*\|^2 \leq R, \quad \text{for all } N\in\N{\setminus}\{0\}.
\end{equation}
Thus proving that the sequences $(z^k)$ and $(u^k)$ are bounded.
\end{proof}


For the convergence result of Algorithm~\ref{alg:pdls} we need the following assumption on the stepsize sequences $(\tau_k)$. Later, in Remark~\ref{remark:backtracking}, we will show that the assumption holds when a backtracking procedure is used for selecting $\tau_k$. 

\begin{assumption}\label{ass:tauk}
The sequence $(\tau_k)$ generated by the linesearch in Algorithm~\ref{alg:pdls} is bounded away from $0$ whenever the sequence $(z^k,u^k)$ is bounded.
\end{assumption}

We are  now in a position to state the convergence result of Algorithm~\ref{alg:pdls}.

\begin{theorem}\label{th:conv}
Suppose that the assumptions of Problem~\ref{problem:mi} hold. In addition, suppose Assumption~\ref{ass:tauk} holds. Suppose there exists at least one solution of the monotone inclusion~(\ref{eq:monotone_inclusion}). Then the sequence $(z^k,u^k)$ generated by Algorithm~\ref{alg:pdls} converges to a primal-dual  solution of~(\ref{eq:monotone_inclusion}).
\end{theorem}
\begin{proof}
By Proposition~\ref{prop:descent_inequality}\ref{it:di-3}, the sequence $(z^k,u^k)$ is bounded. Let $(z^{k_j},u^{k_j})$ be a subsequence of $(z^k,u^k)$ converging to some accumulation point $(\tilde{z},\tilde{u})$. In what follows we prove that $(\tilde{z},\tilde{u})$ is a primal-dual solution of~(\ref{eq:monotone_inclusion}).

Indeed,   the iteration defined by~(\ref{eq:dual_update})-(\ref{eq:primal_update}) along this subsequence can be expressed in matrix form as
\begin{equation}\label{eq:it_mf}
\underbrace{\left[
\begin{array}{cc}
C^{-1} &  -K \\
K^* & A+B
\end{array}
\right]}_{M:=} \left[
\begin{array}{c}
u^{k_j+1} \\
z^{k_j+1}
\end{array}
\right]
\ni
\left[
\begin{array}{l}
\frac{u^{k_j}-u^{k_j+1}}{\tau_{k_j-1}} - K(z^{k_j+1}-z^{k_j})  \\
\frac{z^{k_j}-z^{k_j+1}}{\tau_{k_j}} - \frac{\tau_{k_j-1}}{\tau_{k_j}} K^*(u^{k_j+1}-u^{k_j})-b_{k_j} + \frac{\tau_{k_j-1}}{\tau_{k_j}} b_{k_j-1}
\end{array}
\right],
\end{equation}
where $b_{\ell}:=B(z^{\ell})-B(z^{\ell+1})$. 
Observe that the operator $M$ can be decomposed as the sum of the Cartesian product $C^{-1} \times A$, which is maximally monotone,  with a continuous operator on $\G\times\overline{\dom}\, A$. Since both the graphs of continuous operators and  maximally monotone operators are closed (see, e.g.,~\cite[Proposition~20.38]{bauschke2017convex}) the graph of $M$ is closed on $\G\times\overline{\dom}\, A$.  
By Proposition~\ref{prop:descent_inequality}\ref{it:di-3}, the sequence $(z^k,u^k)$ is bounded, and Assumption~\ref{ass:tauk} yields $\inf_{k\in\N} \tau_k >0$. Together with the uniform upper bound on $(\tau_k)$ on Step~2 of Algorithm~\ref{alg:pdls}, this guarantees that the quotients $1/\tau_{k_j-1}$, $1/\tau_{k_j}$ and $\tau_{k_j-1}/\tau_{k_j}$ in~\eqref{eq:it_mf} remain bounded. Combining all these with Proposition~\ref{prop:descent_inequality}\ref{it:di-2} and recalling the continuity of $B$, by passing to the limit in~(\ref{eq:it_mf}) as $j\to\infty$ the closedness of the graph of $M$ yields 
\[
\left[
\begin{array}{cc}
C^{-1} &  -K \\
K^* & A+B
\end{array}
\right]\left[
\begin{array}{c}
\tilde{u} \\
\tilde{z}
\end{array}
\right]
\ni
\left[
\begin{array}{c}
0 \\
0
\end{array}
\right],
\]
which proves that $(\tilde{z},\tilde{u})$ is a primal-dual solution of~(\ref{eq:monotone_inclusion}).

By Proposition~\ref{prop:descent_inequality}\ref{it:di-1} the sequence $\bigl(\mathcal{V}^{(\tilde{z},\tilde{u})}_k\bigr)$ is nonincreasing and nonnegative due to Lemma~\ref{lemma:cVnonnegative}. Hence, the limit $\lim_{k\to\infty} \mathcal{V}^{(\tilde{z},\tilde{u})}_k$  exists, and since $z^{k}-z^{k-1}\to0$, we have
\[
\begin{aligned}
\lim_{k\to\infty} \mathcal{V}^{(\tilde{z},\tilde{u})}_k  & = \lim_{k\to\infty} \cE^{(\tilde{z},\tilde{u})}_k \\
&  = \lim_{j\to\infty} \cE^{(\tilde{z},\tilde{u})}_{k_j} \\
& =\lim_{j\to\infty}  \frac{1}{2}\big( \|z^{k_j}-\tilde{z} \|^2 + \|u^{k_j}-\tilde{u}\|^2\big) + \tau_{k_j-1} \langle K(z^{k_j}-\tilde{z}), u^{k_j}-\tilde{u} \rangle) =  0.
\end{aligned}
\]
Finally, using the same type of estimation as in~(\ref{eq:bound-3})-(\ref{eq:bound-4}), we  obtain
\[
\lim_{k\to\infty} \cV^{(\tilde{z},\tilde{u})}_k \geq \lim_{k\to\infty} \frac{1}{4} \|z^k -\tilde{z}\|^2 + \frac{1}{4}\|u^k-\tilde{u}\|^2 =0,
\]
which proves that $(z^k,u^k)$ converge to the primal-dual solution $(\tilde{z},\tilde{u})$.
\end{proof}

\begin{remark}[Backtracking linesearch]\label{remark:backtracking}
	The linesearch in Step~2 of Algorithm~\ref{alg:pdls} can be conducted by a backtracking procedure. For instance, for  all $k\in\N$, we can take $\alpha\,{\in}\,{(0,1)}$,  choose a fixed initial trial stepsize  $\bar{\tau} \in \bigl(0,\delta (1-\beta)/(2\|K\|)\bigr)$, and select $\tau_k$ as $\tau_k =\alpha^j \bar{\tau}$ with $j$ being the smallest  integer such that $\tau_k$ satisfies~(\ref{eq:linesearch}). In this case, it is easy to check that Assumption~\ref{ass:tauk} holds. To see this, let us first observe that the set of trial points corresponding to $\tau\in{(0,\bar{\tau}]}$ is bounded for all $k\in\N$ whenever the sequences $(z^k)$ and $(u^k)$  are bounded (which in practice holds thanks to Proposition~\ref{prop:descent_inequality}\ref{it:di-3}). That is, we study the boundedness of the set $Z:=\{ z^{k+1}(\tau) \,: \, \tau\in{(0,\bar{\tau}]}, \, k\in\N \}\subseteq \dom A$, where $z^{k+1}(\tau)$ is defined as in~\eqref{eq:zk1tau}. For all $\tau\in{(0,\bar{\tau}]}$, we obtain the estimation
	\begin{equation}\label{eq:remarkback-1}
		\begin{aligned}
		\bigl\|z^{k+1}(\tau) \bigr\| & \leq  \bigl\|z^{k+1}(\tau) - J_{\tau A} \big(z^{k+1}(0)\big)\bigr\| +  \bigl\|  J_{\tau A} \big(z^{k+1}(0)\big) \bigr\| \\
		& \leq    \bar{\tau} \bigl\|  K^*u^{k+1}+ B(z^k) \bigr\|  + \bigl\|  J_{\tau A} \big(z^{k+1}(0)\big) \bigr\|, 
		\end{aligned}
	\end{equation}
	for all $k\in\N$. Now, let $(a,a^*)\in\gra A$. This in turn implies that $a = J_{\tau A} (a+\tau a^*)$ for all $\tau >0$. Hence, the nonexpansivity of the resolvent and the triangle inequality yield
	\begin{equation}\label{eq:remarkback-2}
	\begin{aligned}
	 \bigl\|  J_{\tau A} \big(z^{k+1}(0)\big) \bigr\| & \leq \bigl\|  J_{\tau A} \big(z^{k+1}(0)\big) - J_{\tau A} \big(a+\tau a^*\big) \bigr\| + \bigl\|  J_{\tau A} \big(a+\tau a^*\big) \bigr\| \\
	 & \leq  \| z^{k+1}(0) - a\| + \bar{\tau} \|a^*\| + \|a\|, 
	\end{aligned}
	\end{equation}
	for all $\tau\in{(0,\bar{\tau}]}$ and all $k\in\N$. Since $(z^k)$ and $(u^k)$ are bounded, we conclude that the sequence $\bigl(z^{k+1}(0)\bigr)$ is also bounded (recall~\eqref{eq:zk10}), which together with \eqref{eq:remarkback-1}-\eqref{eq:remarkback-2} ensures that the set $Z$ is bounded. Hence, we can take a compact convex set $U\subseteq\overline{\dom}\,A$ containing $\{z^k\} \cup Z$ such that $B$ is Lipschitz continuous on $U$ with Lipschitz constant $L\geq0$.  Thus, for all $k\in\N$, we can conclude that the inequality
	\[
	\tau_k \| B\bigl(z^{k+1}(\tau_k)\bigr) - B\bigl(z^k\bigr)\| \leq \frac{\sqrt{\beta}}{2}\|z^{k+1}(\tau_k) - z^k\|,
	\]
	holds for some $\tau_k \in{[\alpha^{\ell} \bar{\tau},\bar{\tau}]}$, with $\ell$ being the smallest integer such that  $\alpha^{\ell}\bar{\tau} \leq \sqrt{\beta}/(2L)$.
\end{remark}

 The  previous remark leads to the following corollary of Theorem~\ref{th:conv}.
 
 \begin{corollary}
 Suppose that the assumptions of Problem~\ref{problem:mi} hold.  Suppose there exists at least one solution of the monotone inclusion~(\ref{eq:monotone_inclusion}). Let  $(z^k,u^k)$ be the sequence  generated by Algorithm~\ref{alg:pdls} where the linesearch is performed following a backtracking procedure as detailed in Remark~\ref{remark:backtracking}. Then $(z^k,u^k)$ converges to a primal-dual  solution of~(\ref{eq:monotone_inclusion}). 
 \end{corollary}

\subsection{Ergodic convergence rates}

In this section, we focus on the consensus min-max problems~\eqref{eq:minimax_pd0}.  
In particular, we obtain ergodic convergence rates for a gap function associated with this problem. 
We start by recalling the primal variable $\bz:=(\bx,\by)\in\R^{n\times p}\times\R^{n\times q}$ of~\eqref{eq:mm_as_mi}. In addition, we will use  $\bu\in\R^{n\times p}\times\R^{n\times q}$ for the dual variable for this problem. In accordance with these variables, we will also employ boldface for the sequence $(\bz^k,\bu^k)$ of primal-dual variables generated by Algorithm~\ref{alg:pdls} applied to~\eqref{eq:mm_as_mi}.

The following notation will be used throughout this section. For $\bz:=(\bx,\by)\in\R^{n\times p}\times\R^{n\times q}$, we define
\begin{equation*}
    h(\bz) := f(\bx) + g(\by) \quad\text{and}\quad\Phi(\bx,\by) := f(\bx) + \phi(\bx,\by) - g(\by),
\end{equation*}
where $f,g$ and $\phi$ are defined as in~\eqref{eq:fgphi}. Observe that $\partial h = A$, where $A$ is the maximally monotone operator defined in $\eqref{eq:opreform}$, while $\Phi$ is the saddle function in~\eqref{eq:minimax_cons}.
Given a reference point $\tilde{\bz}:=(\tilde{\bx},\tilde{\by})\in\mathcal{D}_{\R^{n\times p}}\times\mathcal{D}_{\R^{n\times q}}$ and a vector $\tilde{\bu}\in\R^{n\times p}\times\R^{n\times q}$, we consider the \emph{consensus-penalised gap function} \begin{equation}\label{eq:gapfunction}
\mathcal{G}(\bz;\tilde{\bz},\tilde{\bu}) := \Phi(\bx,\tilde{\by}) - \Phi(\tilde{\bx},\by) + \langle K\bz, \tilde{\bu}\rangle, \quad \bz:=(\bx,\by)\in\R^{n\times p}\times\R^{n\times q}.
\end{equation} 
The following lemma shows that the consensus-penalised gap is nonnegative when the reference points are given by a primal-dual solution of~\eqref{eq:mm_as_mi}. The proof is included for completeness.

\begin{proposition}\label{prop:cGnonnegative}
Let $(\bz^*,\bu^*)$ be a primal-dual solution in the sense of~\eqref{eq:pdincl} for the monotone inclusion problem~\eqref{eq:mm_as_mi}. Then 
\begin{equation*}
    \mathcal{G}(\bz;\bz^*,\bu^*)  \geq 0, \quad \text{for all } \bz=(\bx,\by)\in\dom h.  
\end{equation*}
\end{proposition}
\begin{proof}
The  primal-dual inclusion~\eqref{eq:pdincl} with the operators in~\eqref{eq:opreform} reduces to 
\begin{equation}\label{eq:pdiG}
\left\{
\begin{aligned}
  &0  = K\bz^*, \\
 &\partial f(\bx^*)\ni-K_1^* \bu_x^* - \nabla_{\bx} \phi(\bx^*,\by^*) , \\
 &\partial g(\by^*)\ni- K_2^*\bu_y^* + \nabla_\by\phi(\bx^*,\by^*)   ,
\end{aligned}
\right.
\end{equation}
where $\bu^*=(\bu_x^*,\bu_y^*)$. Let $\bz=(\bx,\by)\in\dom h$. The convexity of $f$ and $g$ then yields
\begin{equation}\label{eq:Gapf-1}
f(\bx)-f(\bx^*) \geq  \langle K_1^*\bu^*_x + \nabla_\bx\phi(\bx^*,\by^*), \bx^*-\bx \rangle,  \quad g(\by )- g(\by^*) \geq  \langle K_2^*\bu^*_y - \nabla_\by\phi(\bx^*,\by^*), \by^*-\by \rangle.
\end{equation}
Furthermore, by the convexity of $\phi(\cdot,\by^*)$ and the concavity of $\phi(\bx^*,\cdot)$, we have \begin{equation*}
        \phi(\bx,\by^*) - \phi(\bx^*,\by^*) \geq \langle \nabla_{\bx}\phi(\bx^*,\by^*), \bx-\bx^*\rangle, \quad \phi(\bx^*,\by^*) - \phi(\bx^*,\by) \geq \langle - \nabla_{\by}\phi(\bx^*,\by^*), \by-\by^*\rangle.
\end{equation*}
By adding the two equations above, we get
\begin{equation}\label{eq:Gapf-2}
\phi(\bx,\by^*)  - \phi(\bx^*,\by)  \geq \langle \nabla_{\bx}\phi(\bx^*,\by^*), \bx-\bx^*\rangle + \langle - \nabla_{\by}\phi(\bx^*,\by^*), \by-\by^*\rangle.
\end{equation}
Therefore, for all $\bz\in\dom h$, making use of the relations~\eqref{eq:Gapf-1} and~\eqref{eq:Gapf-2} we obtain
\[
\begin{aligned}
 \mathcal{G}(\bz;\bz^*,\bu^*)  &= \Phi(\bx,\by^*)  - \Phi(\bx^*,\by) + \langle K\bz, \bu^*\rangle \\
 & \geq \langle K_1^* \bu^*_x , \bx^*-\bx\rangle + \langle K_2^*\bu^*_y, \by^*-\by\rangle + \langle K\bz,\bu^*\rangle\\
 & = \langle K\bz^*,\bu^*\rangle = 0,
\end{aligned}
\]
where the last equality is due to the first equation in~\eqref{eq:pdiG}.
\end{proof}

Below, we rely on the following technical lemma, which is a specialisation of Lemma~\ref{l:intermediate} to our current setting.

\begin{lemma}\label{l:gap-estimate}
    Suppose that the functions $f_i, g_i$ and $\phi_i$ satisfy the assumptions of Problem~\ref{problem:1}, for all $i\in\{1,2,\ldots,n\}$. Let $\tilde{\bz}=(\tilde{\bx},\tilde{\by})\in\bigl(\mathcal{D}_{\R^{n\times p}}\times\mathcal{D}_{\R^{n\times q}}\bigr)\cap\dom h$ and $\tilde{\bu}\in\R^{n\times p}\times\R^{n\times q}$. Let $(\bz^k)$ and $(\bu^k)$ be the sequences generated by Algorithm~\ref{alg:pdls} applied to~\eqref{eq:mm_as_mi}. Then, for all $k\in\N$, it holds\begin{equation}\label{eq:Ginequality}
        \begin{aligned}
\tau_k \mathcal{G}(\bz^{k+1};\tilde{\bz},\tilde{\bu}) + \cV_{k+1}^{(\tilde{\bz},\tilde{\bu})} & \leq \cV_k^{(\tilde{\bz},\tilde{\bu})} - \frac{1}{2}\left( \frac{1-\beta}{2} -  \tau_{k-1}\|K\|\right) \|\bz^{k+1}-\bz^k\|^2 \\
& \quad - \frac{1}{2}(1-\tau_{k-1} \|K\|) \|\bu^{k+1}-\bu^k\|^2.
\end{aligned}
    \end{equation}
\end{lemma}
\begin{proof}
    By Lemma~\ref{l:intermediate}, it suffices to show that when specialised to problem~\eqref{eq:mm_as_mi} the residual $\cR_k^{(\tilde{\bz},\tilde{\bu})}$ is lower bounded by $\tau_k\cG(\bz^{k+1};\tilde{\bz},\tilde{\bu})$, for all $k\in\N$.

    Fix $k\in\N$. Since $C=N_{\{0\}}$, the element $\omega^{k+1} \in C^{-1}(u^{k+1}) $ in~\eqref{eq:c-graph_elements} vanishes and $K\tilde{\bz}=0$ by assumption, then the second summand in the residual~\eqref{eq:residual} becomes zero. Thus, we get \begin{equation*}
        \cR_k^{(\tilde{\bz},\tilde{\bu})} = \tau_k \bigl\langle \bz^{k+1}-\tilde{\bz}, \, \nu^{k+1}\bigr\rangle + \tau_k \bigl\langle K\bz^{k+1}, \tilde{\bu}\bigr\rangle + \tau_k \bigl\langle B(\bz^{k+1}), \bz^{k+1}-\tilde{\bz} \bigr\rangle.
    \end{equation*}
    The subdifferential inequality for $A = \partial h$ then yields $\langle \bz^{k+1}-\tilde{\bz}, \nu^{k+1}\rangle \geq h(\bz^{k+1}) - h(\tilde{\bz})$.  Next, using the convexity of $\phi(\cdot,\by^{k+1})$ and the concavity of $\phi(\bx^{k+1},\cdot)$, we get the inequalities
\[
\phi(\bx^{k+1},\by^{k+1})-\phi(\tilde{\bx},\by^{k+1}) \leq\langle \nabla_\bx\phi(\bx^{k+1},\by^{k+1}),\bx^{k+1}-\tilde{\bx}\rangle 
\]
and
\[
 \phi(\bx^{k+1},\tilde{\by})-\phi(\bx^{k+1},\by^{k+1}) \leq \langle \nabla_\by\phi(\bx^{k+1},\by^{k+1}),\tilde{\by}-\by^{k+1}\rangle.
\]
Combining both inequalities yields
\begin{equation*}
    \langle B(\bz^{k+1}) , \bz^{k+1}-\tilde{z} \rangle \geq  \phi(\bx^{k+1},\tilde{\by})-\phi(\tilde{\bx},\by^{k+1}).
\end{equation*}
    Altogether, we conclude that \begin{equation*}
        \cR_k^{(\tilde{\bz},\tilde{\bu})} \geq \tau_k \bigl( h(\bz^{k+1}) - h(\tilde{\bz}) + \phi(\bx^{k+1},\tilde{\by}) - \phi(\tilde{\bx},\by^{k+1}) + \langle K\bz^{k+1},\tilde{\bu}\rangle \bigr) = \tau_k \mathcal{G}(\bz^{k+1}; \tilde{\bz},\tilde{\bu}),
    \end{equation*} which yields the desired result.
\end{proof}

We are now in a position to state our result on ergodic convergence rates. To this end, for all $N\in\N$, we define the ergodic sequences \begin{equation}\label{eq:ergodic-squences}
    S_N := \sum_{k=0}^N \tau_k \quad \text{and} \quad \widehat{\bz}^N = \bigl(\hbx^N, \hby^N\bigr) := \frac{1}{S_N}\sum_{k=0}^N \tau_k\, \bigl(\bx^{k+1},\by^{k+1}\bigr),
\end{equation}
where $(\bz^k)=(\bx^k,\by^k)$ and $(\tau_k)$ stand for the sequences generated by Algorithm~\ref{alg:pdls} applied to~\eqref{eq:mm_as_mi}.

\begin{theorem}\label{th:first-ergodic}
Suppose that the functions $f_i, g_i$ and $\phi_i$ satisfy the assumptions of Problem~\ref{problem:1} for all $i\in\{1,2,\ldots,n\}$.  Let $(\widehat{\bz}^N)$  be the ergodic sequence defined in~\eqref{eq:ergodic-squences}. The following assertions hold.
\begin{enumerate}[label=(\roman*)]
\item\label{it:t-first-erg-1} 
For every $\tilde{\bz}=(\tilde{\bx},\tilde{\by})\in\bigl(\mathcal{D}_{\R^{n\times p}}\times\mathcal{D}_{\R^{n\times q}}\bigr)\cap\dom h$ and $\tilde{\bu}\in\R^{n\times p}\times\R^{n\times q}$, and for all $N\in\N$, it holds
\begin{equation}\label{eq:ergodic-gap}
        \mathcal{G}(\widehat{\bz}^N;\tilde{\bz},\tilde{\bu})\leq \frac{\cV_0^{(\tilde{\bz},\tilde{\bu})}}{S_N}.
    \end{equation}
    
\item\label{it:t-first-erg-2} In addition, suppose that~\eqref{eq:mm_as_mi} has at least one solution and that Assumption~\ref{ass:tauk} holds. Then the sequence $(\widehat{\bz}^N)$  converges to $\bz^*\in\mathcal{D}_{\R^{n\times p}}\times\mathcal{D}_{\R^{n\times q}}$ as $N\to\infty$, with $\bz^*$ being the solution of~\eqref{eq:mm_as_mi} to which the sequence $(\bz^k)$ generated by Algorithm~\ref{alg:pdls} converges by Theorem~\ref{th:conv}.

\item\label{it:t-first-erg-3} Furthermore, for all $N\in\N$, we get
    \begin{align}\label{eq:ergrates}
    \bigl\|K\widehat{\bz}^N\bigr\| & = \mathcal{O}\left( \frac{1}{N}\right) \quad\text{and} \quad\left| \Phi(\widehat{\bx}^N,\bar{\by}) - \Phi(\bar{\bx},\widehat{\by}^N)\right| = \mathcal{O}\left( \frac{1}{N}\right),
\end{align}
for any  primal solution $\bar{\bz}=(\bar{\bx},\bar{\by})\in\mathcal{D}_{\R^{n\times p}}\times\mathcal{D}_{\R^{n\times q}}$  of~\eqref{eq:mm_as_mi}. 
\end{enumerate}
\end{theorem}
\begin{proof}
\ref{it:t-first-erg-1} Let $N\in\N$. Summing inequality~\eqref{eq:Ginequality} from $k=0$ to $N$, dropping the nonpositive terms in the right-hand side (recall that $\tau_{k-1}\|K\|<\delta(1-\beta)/2$) and using the nonnegativity of $\cV_{N+1}^{(\tilde{\bz},\tilde{\bu})}$ (recall Lemma~\ref{lemma:cVnonnegative}), we get \begin{equation*}
        \sum_{k=0}^N \tau_k\mathcal{G}(\bz^{k+1};\tilde{\bz},\tilde{\bu}) \leq \cV_0^{(\tilde{\bz},\tilde{\bu})} - \cV_{N+1}^{(\tilde{\bz},\tilde{\bu})} \leq \cV_0^{(\tilde{\bz},\tilde{\bu})}.
    \end{equation*} Dividing by $S_N$ and using the convexity of $\mathcal{G}(\cdot\,;\tilde{\bz},\tilde{\bu})$, then \eqref{eq:ergodic-gap} follows. 

    \ref{it:t-first-erg-2} By Proposition~\ref{prop:descent_inequality}\ref{it:di-3}, the sequence $(\bz^k,\bu^k)$ is bounded, and then Assumption~\ref{ass:tauk} yields, for all $k\in\N$, $\tau_k\geq\tau_{\min}$ for some $\tau_{\min}>0$. 
    In particular, $S_N\geq (N+1)\tau_{\min}\to\infty$. Moreover, the upper bound in Step~2 of Algorithm~\ref{alg:pdls} yields  $\tau_k< \delta(1-\beta)/(2\|K\|)$, for all $k\in\N$. Let $\varepsilon>0$.  By Theorem~\ref{th:conv}, there exists $N_\varepsilon\in\N$ such that $\|\bz^{k+1}-\bz^*\|\leq2\tau_{\min}\varepsilon\|K\|/(\delta(1-\beta))$ for all $k\geq N_\varepsilon$. Hence, for all $N\geq N_\varepsilon$,
\[
\begin{aligned}
\|\widehat{\bz}^N-\bz^*\| &\leq \frac{1}{S_N}\sum_{k=0}^N\tau_k\|\bz^{k+1}-\bz^*\| \\
& \leq \frac{1}{S_N}\sum_{k=0}^{N_\varepsilon-1}\tau_k\|\bz^{k+1}-\bz^*\| + \frac{1}{S_N}\sum_{k=N_\varepsilon}^N\tau_k\|\bz^{k+1}-\bz^*\|\\
&< \frac{1}{(N+1)\tau_{min}}\sum_{k=0}^{N_\varepsilon-1}\tau_k\|\bz^{k+1}-\bz^*\| + \varepsilon, 
\end{aligned}
\]
and letting $N\to\infty$ gives $\limsup_{N\to\infty}\|\widehat{\bz}^N-\bz^*\|\leq\varepsilon$. Since $\varepsilon>0$ was arbitrary, $\widehat{\bz}^N\to\bz^*$. Finally, recall that $\bz^*\in\mathcal{D}_{\R^{n\times p}}\times\mathcal{D}_{\R^{n\times q}}$ since $K\bz^*=0$.

    \ref{it:t-first-erg-3} We now show the first estimate in~\eqref{eq:ergrates}. If $K\widehat{\bz}^N=0$ there is nothing to prove. Otherwise, let $(\bar{\bz},\bar{\bu})$ be a primal-dual solution of~\eqref{eq:mm_as_mi} and set $\tilde{\bu}:=\bar{\bu}+K\widehat{\bz}^N/\|K\widehat{\bz}^N\|$ to get \[\mathcal{G}(\widehat{\bz}^N;\bar{\bz},\tilde{\bu})=\mathcal{G}(\widehat{\bz}^N;\bar{\bz},\bar{\bu})+\|K\widehat{\bz}^N\|.\] Set $R_0:=\max\bigl\{\cV_0^{(\bar{\bz},\bu)}\, :\, \|\bu-\bar{\bu}\|\leq 1\bigr\}\geq0$. As $(\bar{\bz},\bar{\bu})$ is a primal-dual solution, $\mathcal{G}(\widehat{\bz}^N;\bar{\bz},\bar{\bu}) \geq 0$ by Proposition~\ref{prop:cGnonnegative}. It then follows from \eqref{eq:ergodic-gap} that \begin{equation*}
        \|K\widehat{\bz}^N\| \leq \mathcal{G}(\widehat{\bz}^N;\bar{\bz},\tilde{\bu}) \leq \frac{\cV_0^{(\bar{\bz},\tilde{\bu})}}{S_N} \leq \frac{R_0}{S_N} \leq \frac{R_0}{(N+1)\tau_{\min}},
    \end{equation*} which proves the desired rate. For the second estimate in~\eqref{eq:ergrates}, take $\tilde{\bz}:=\bar{\bz}$ and $\tilde{\bu}:=0$ in \eqref{eq:ergodic-gap} to obtain \begin{equation}\label{eq:ecrPhi-1} 
    \Phi\bigl(\hbx^N, \bar{\by}\bigr)  - \Phi\bigl(\bar{\bx}, \hby^N\bigr) \leq \frac{R_1}{(N+1)\tau_{\min}},
    \end{equation} where $R_1:=\cV_0^{(\bar{\bz},0)}\geq0$. Moreover, using the Cauchy-Schwarz inequality and the estimate above for $\| K\widehat{\bz}^N\|$, 
    \begin{equation}\label{eq:ecrPhi-2} \begin{aligned}
        0 &\leq \mathcal{G}(\widehat{\bz}^N;\bar{\bz},\bar{\bu}) \\
        &= \Phi(\widehat{\bx}^N,\bar{\by}) - \Phi(\bar{\bx},\widehat{\by}^N) + \langle K\widehat{\bz}^N,\bar{\bu}\rangle\\
         &\leq \Phi(\widehat{\bx}^N,\bar{\by}) - \Phi(\bar{\bx},\widehat{\by}^N) + \frac{R_0}{(N+1)\tau_{\min}} \|\bar{\bu}\|.
    \end{aligned} 
    \end{equation}
    Hence, \eqref{eq:ecrPhi-1} and~\eqref{eq:ecrPhi-2} conclude the second estimate in~\eqref{eq:ergrates}
\end{proof}

\section{Distributed proximal-reflected-gradient  algorithm with  linesearch for min-max problems}\label{sect:distalgorithm}

We now turn our attention to deriving a distributed algorithm for Problem~\ref{problem:1}. This will be achieved by an appropriate application of Algorithm~\ref{alg:pdls} to the problem reformulation outlined in Section~\ref{sect:rmi}.

Let us recall the primal variable $\bz\in\R^{n\times p} \times \R^{n\times q}$ of problem~\eqref{eq:mm_as_mi}, the operators $A,B$ and the matrix $K$ defined in~\eqref{eq:opreform}. In addition, define the dual variable $\bu\in\R^{n\times p} \times \R^{n\times q}$. By noting that $C=N_{\{0\}}$, \cite[Corollary 16.3]{bauschke2017convex} yields that  $J_{C^{-1}} = J_{0} =  \Id$ and the update rules~\eqref{eq:dual_update}-\eqref{eq:primal_update} in Algorithm~\ref{alg:pdls} applied to~\eqref{eq:mm_as_mi} can be written as
\begin{equation}\label{eq:dist_it}
\left\{
\begin{aligned}
\bu^{k+1} &=  \bu^k + \tau_{k-1}K\bz^k,\\
\bw^{k+1} & = \bz^k - \tau_k K^*\bu^{k+1} - \tau_{k-1} K^*(\bu^{k+1}-\bu^k) - \tau_k B(\bz^k) - \tau_{k-1} \bigl(B(\bz^k)-B(\bz^{k-1})\bigr), \\
\bz^{k+1} &=  J_{\tau_k A} (\bw^{k+1} ).
\end{aligned}
\right.
\end{equation}

Note that the subroutines related to the linesearch are omitted in~\eqref{eq:dist_it}. We will later address these details and discuss a procedure that allows for a distributed linesearch implementation, but for now we will concentrate on ensuring that the updates described by iteration~\eqref{eq:dist_it} are performed in a decentralised fashion. 
To this end, observe that  information exchange between agents occurs exclusively when evaluating the linear operators $K$ and $K^*$. To better understand the role of these matrix-vector operations in the communication pattern, let us take as a representative example the $\bx$ component in the product $K\bz=(K_1\bx,K_2\by)$. The premultiplication of the vector $\bx\in\R^{n\times p}$ by the matrix $K_1=(K_{1,ij})_{i,j=1}^n$ leads to a vector $K_1 \bx \in \R^{n\times p}$ whose $i$-th row is given by
\[
(K_1 \bx)_i = \sum_{j=1}^n K_{1,ij} x_j^\top, \quad \text{for all } i\in\{1,2,\ldots,n\}.
\]
Thus, the new vectors given by the rows of $K_1 \bx$ are obtained as linear combinations of  the agent vectors $x_1, x_2, \ldots, x_n$. We aim to adapt these communication rounds to the desired network topology, whose structure we assume to be determined by the (undirected) network graphs  $\cG_1:= (\cV_1,\cE_1)$ and $\cG_2:=(\cV_2,\cE_2)$,  for the variables $\bx$ and $\by$, respectively. Ideally, information exchange should be kept to a minimum---preferably to a single round per iteration. Distributed algorithms typically achieve this by leveraging \emph{mixing matrices}.

\begin{definition}[Mixing matrix]\label{def:mm}
	Let $\cG:=(\cV,\cE)$ be an undirected connected graph with  $\left| \cV \right| =n$ nodes. A matrix $W\in\R^{n\times n}$ is said to be a \emph{mixing matrix} if
	\begin{enumerate}[label=(\roman*)]
		\item\label{it:mm-1} $W_{ij} = 0$, whenever  $(i,j) \notin \cE$;
		\item\label{it:mm-2} $W=W^T$;
		\item\label{it:mm-3} $\ker(I-W) = \mathcal{D}_{\mathbb{R}^{n}}$;
		\item\label{it:mm-4} $ W \preceq I$, where $I$ is the identity matrix.\footnote{Usually mixing matrices are also required to also satisfy the condition $-I \prec W$, but this is not required in our analysis.}
	\end{enumerate}
\end{definition}

Generally, one can construct a mixing matrix of a graph $\cG$ as $I-\mathcal{L}/\sigma$ with $\mathcal{L}$ the Laplacian matrix and $\sigma>\lambda_{max}(\mathcal{L})/2$. Further examples can be consulted, for instance, in~\cite[Section~2.4]{MR3343366}.
Property~\ref{it:mm-1} ensures that only connected agents can exchange information. Conditions~\ref{it:mm-2}-\ref{it:mm-4} are used to define a proper consensus matrix that allows for a single communication round in an algorithmic iteration. This will become clear shortly.

Recall that every agent in our framework is in charge of both variables $x_i$ and $y_i$. We will allow different network topologies for each of these variables. That is, the graphs  $\cG_1$ and $\cG_2$ associated  to the variables $\bx$ and $\by$ respectively  are not necessarily the same. Hence, two different mixing matrices will be considered for each of such graphs. For $t\in\{1,2\}$, let $W_t$  be a mixing matrix of $\cG_t$. We set $K_t = \bigl(I-W_t\bigr)^{1/2}$, which is well-defined by conditions~\ref{it:mm-2} and~\ref{it:mm-4}. In particular, $K$ in~\eqref{eq:dist_it} will be taken as $K=(K_1,K_2)$ to obtain a satisfactory distributed implementation of our algorithm. The following proposition presents some useful properties of $K$ that  motivate this choice.

\begin{proposition}\label{prop:K}
The following properties hold:
\begin{enumerate}[label=(\roman*)]
\item\label{it:K1} $K^*=K$ and $KK=(I-W)$ with $W:=(W_1,W_2)$;
\item\label{it:K2} $\ker K = \mathcal{D}_{\R^{n}}\times\mathcal{D}_{\R^{n}}$;
\item\label{it:K3} $\|K\| = \sqrt{1- \min\bigl\{\lambda_{\min}(W_1),\lambda_{\min}(W_2)\bigr\}}$.
\end{enumerate}
\end{proposition}
\begin{proof}
\ref{it:K1} It easily follows by construction. \ref{it:K2} By  the definition of the square root of a positive semidefinite matrix and Definition~\ref{def:mm}\ref{it:mm-3},  $\ker K_i = \ker (I-W_i)^{1/2} = \ker (I-W_i)$. Let $\bz=(\bx,\by)\in\R^{n\times p}\times\R^{n\times q}$. Then $\bz \in\ker{K} \Longleftrightarrow (\bx,\by)\in\ker (I-W_1) \times \ker (I-W_2) \Longleftrightarrow (\bx,\by)\in \mathcal{D}_{\R^{n}}\times \mathcal{D}_{\R^{n}}$.  \ref{it:K3} Finally, to compute the norm of $K$ it suffices to notice that this operator can be written in matrix form as 
\[
K = 
\begin{bmatrix}
	(I-W_1)^{1/2} & 0 \\
	0 & (I-W_2)^{1/2} 
\end{bmatrix}.
\]
Hence, its norm is given by its maximum eigenvalue, which is simply
\[
\|K\| = \max \bigl\{ \lambda_{\max}\bigl( (I-W_1)^{1/2}\bigr), \lambda_{\max} \bigl( (I-W_2)^{1/2} \bigr) \bigr\} = \sqrt{1 - \min\bigl\{\lambda_{\min}(W_1),\lambda_{\min}(W_2)\bigr\}},
\]
concluding the proof.
\end{proof}

By Proposition~\ref{prop:K}\ref{it:K2}, the linear operator $K$ enforces consensus on both variables $\bx$ and $\by$. In addition, Proposition~\ref{prop:K}\ref{it:K1} allows us to obtain an iteration equivalent to~\eqref{eq:dist_it} that requires only a single communication round.  
 Indeed, by defining the new variable $\bv^k=K\bu^k$, premultiplying the first equation in~\eqref{eq:dist_it} by $K$, and performing the appropriate change of variables, we get the equivalent iteration
\begin{equation}\label{eq:final_dist_it}
	\left\{
	\begin{aligned}
		\bv^{k+1} &=  \bv^k + \tau_{k-1}(I-W)\bz^k, \\
		 \bw^{k+1} &=   \bz^k - \tau_k \bv^{k+1} - \tau_{k-1}^2(I-W)\bz^k - \tau_k B(\bz^k) - \tau_{k-1} \bigl(B(\bz^k)-B(\bz^{k-1})\bigr), \\
		\bz^{k+1} &=   J_{\tau_k A}(\bw^{k+1}),
	\end{aligned}
	\right.
\end{equation}
where by the definition of the sequence $(\bv^k)$ it is required that $\bv^0\in\range K$. The latter trivially holds if $\bv^0=0$.
Consequently, \eqref{eq:final_dist_it} describes a distributed update where only a single communication round is performed per iteration to transmit the information encoded by the vector $\bz^k$.

The last step to complete our distributed algorithm for Problem~\ref{problem:1} is to tailor the linesearch in Step~2 of Algorithm~\ref{alg:pdls} to agent-wise operations. To this end, we follow the linesearch proposed in~\cite[Section~4.2]{atenas2024distributed} and devise a linesearch procedure where the $n$ agents communicate to conduct a global scalar sum. The distributed algorithm resulting from combining iteration~\eqref{eq:final_dist_it} with our proposed linesearch procedure is presented in Algorithm~\ref{alg:dist}. For conciseness, we have included the auxiliary sequences $(\bs_x^k)$ and $(\bs_y^k)$. Clearly, these variables agree with our distributed implementation, as each agent locally computes their corresponding component in these vectors.

\begin{algorithm}
	\caption{Distributed proximal-reflected-gradient  algorithm with  linesearch (DPRGL) for Problem~\ref{problem:1}}\label{alg:dist}
	
	Choose mixing matrices $W_1,W_2\in\R^{n\times n}$\;
	Choose  parameters $\delta\in{(0,1)}$, $\beta\in{(0,1)}$ and $\alpha\in{(0,1)}$\;
	Choose a positive lower bound for the initial trial stepsizes {\small$\munderbar{\tau} < \delta \frac{1-\beta}{2\sqrt{1- \min\{\lambda_{\min}(W_1),\lambda_{\min}(W_2)\}}}$}\;
	Choose an initial trial stepsize $\tau_{-1} \in{\left(\munderbar{\tau}, \delta \frac{1-\beta}{2\sqrt{1- \min\{\lambda_{\min}(W_1),\lambda_{\min}(W_2)\}}}\right)}$\;
	Choose initial points $\bx^{-1}=\bx^{0}\in\dom f$, $\by^{-1}=\by^0\in\dom g$, $\bv^{0}_x=0$, $\bv^{0}_y=0$ and  set
	\[ 
	\bs^{-1}_x :=   \nabla_x \phi(\bx^{-1},\by^{-1}),
	\quad
	\bs_y^{-1} :=  -\nabla_y \phi(\bx^{-1}, \by^{-1});
	\]
	\For{$k=0,1,\ldots$}
	{
		
		$\rhd$ \textbf{Step 1} (Dual update)
		\For{$i=1,2,\ldots,n$}
		{
			\[
			\left\{\begin{aligned}
				v_{x,i}^{k+1} = v_{x,i}^k + \tau_{k-1}\bigl(x_i^k - \sum_{j=1}^n W_{1,ij}x_j^k\bigr), \\
				v_{y,i}^{k+1} = v_{y,i}^k + \tau_{k-1}\bigl(y_i^k - \sum_{j=1}^n W_{2,ij}y_j^k\bigr).
			\end{aligned}\right.
			\]
		}
		
		$\rhd$ \textbf{Step 2} (Linesearch and primal update). Take {\scriptsize$\tau_{k(0)} \in {\left(\munderbar{\tau}, \delta \frac{1-\beta}{2\sqrt{1- \min\{\lambda_{\min}(W_1),\lambda_{\min}(W_2)\}}}\right)}$} and set 
		\[
		\bs^{k}_x :=   \nabla_x \phi(\bx^{k},\by^{k}),
		\quad
		\bs_y^{k} :=  -\nabla_y \phi(\bx^{k}, \by^{k}).
		\]
		\For{$\ell=0,1,\ldots$}
		{
			
				$\rhd$ \textbf{Step 2.a} (Primal agent-wise operations) 
				\For{$i=1,2,\ldots,n$}
				{Compute
				\[
				\footnotesize
				\left\{
				\begin{aligned}
					x^{k(\ell)+1}_i & =  \prox_{\tau_{k(\ell)} f_i} \biggl(x^k_i - \tau_{k(\ell)} v_{x,i}^{k+1}   - \tau_{k-1}^2 \bigl(x^k_i - \sum_{j =1 }^n W_{1,i j} x^k_j\bigr)  - \tau_{k(\ell)} s^k_{x,i} - \tau_{k-1}(s^k_{x,i}-s^{k-1}_{x,i} ) \biggr), \\
					y^{k(\ell)+1}_i & =  \prox_{\tau_{k(\ell)} g_i} \biggl(y^k_i - \tau_{k(\ell)} v_{y,i}^{k+1}  - \tau_{k-1}^2 \bigl(y_i^k - \sum_{j=1}^n W_{2,ij}y_j^k \bigr) - \tau_{k(\ell)} s^k_{y,i} - \tau_{k-1}(s^k_{y,i}-s^{k-1}_{y,i} ) \biggr). \\
				\end{aligned}
				\right.
				\]
				and the auxiliary scalar quantities
				\begin{equation}\label{eq:aici}
					\left\{
					\begin{aligned}
						a_i  &:=  \tau_{k(\ell)}^2 \|\nabla_x \phi_i (x^{k(\ell)+1}_i,y^{k(\ell)+1}_i)- s^k_{x,i}\|^2 -  \frac{\beta}{4}  \|x^{k(\ell)+1}_i-x^{k}_i\|^2, \\
						c_i &:=    \tau_{k(\ell)}^2 \|-\nabla_y \phi_i (x^{k(\ell)+1}_i,y^{k(\ell)+1}_i)- s^k_{y,i}\|^2 -  \frac{\beta}{4}  \|y^{k(\ell)+1}_i-y^{k}_i\|^2.
					\end{aligned}
					\right.
				\end{equation}
				} 
				$\rhd$ \textbf{Step 2.b} (Check the linesearch condition). Compute the global sum $\sum_{i=1}^n (a_i+c_i)$, and  if 
				\begin{equation}\label{eq:back-cond-alg2}
				\sum_{i=1}^n (a_i+c_i) >0,
				\end{equation}
				set $\tau_{k(\ell+1)} = \alpha \tau_{k(\ell)}$ and go back to Step~2.a. Otherwise set $\tau_{k} =\tau_{k(\ell)}$, $\bx^{k+1}= \bx^{k(\ell)+1}$,  $\by^{k+1}= \by^{k(\ell)+1}$ and go back to Step 1.
		}
	}
\end{algorithm}

Specifically, the linesearch in Algorithm~\ref{alg:dist} is a backtracking-type procedure consisting of two stages. First, in Step~2.a, every agent computes locally its contribution to the global sum in the linesearch. Secondly, this global scalar sum is computed in Step~2.b. The following result formalises the connection between Algorithm~\ref{alg:pdls} and Algorithm~\ref{alg:dist}. In addition, it ensures that the linesearch in Algorithm~\ref{alg:dist} satisfies Assumption~\ref{ass:tauk}.

\begin{proposition}\label{prop:dist_ls}
	Suppose that the functions $f_i, g_i$ and $\phi_i$ satisfy the assumptions of Problem~\ref{problem:1}, for all $i\in\{1,2,\ldots,n\}$. Then the following assertions hold.
	\begin{enumerate}[label=(\roman*)]
		\item\label{item:algsr-1} The sequence $(\bz^k)=(\bx^k,\by^k)$ generated by Algorithm~\ref{alg:dist} coincides with that of Algorithm~\ref{alg:pdls} applied to~\eqref{eq:mm_as_mi} initialised with $\bu^0=0$ and under a backtracking linesearch.
		\item\label{item:algsr-3} Furthermore, the sequence $(\tau_k)$ generated by Algorithm~\ref{alg:dist} satisfies Assumption~\ref{ass:tauk}.
	\end{enumerate}
\end{proposition}

\begin{proof}
	\ref{item:algsr-1} Let us start by formalising the iterative scheme of Algorithm~\ref{alg:pdls} under the conditions in~\ref{item:algsr-1}. A general iteration of Algorithm~\ref{alg:pdls} with a backtracking linesearch applied to~\eqref{eq:mm_as_mi} is written in compact form as~\eqref{eq:dist_it} where the stepsize $\tau_k$ is given by $\tau_{k(\ell)} := \alpha^{\ell} \tau_{k(0)}$ with $\tau_{k(0)}\in{\bigl(0,\delta (1-\beta)/(2\sqrt{1- \min\{\lambda_{\min}(W_1),\lambda_{\min}(W_2)\}})\bigr)}$ (since $\|K\|=\sqrt{1- \min\{\lambda_{\min}(W_1),\lambda_{\min}(W_2)\}}$ by Proposition~\ref{prop:K}\ref{it:K3}) and where $\ell$ is the smallest integer such that 
	\begin{equation}\label{eq:ls-alg1-mi}
	\tau_{k(\ell)} \|B(\bz^{k(\ell)+1} ) - B(\bz^k)\| \leq \frac{\sqrt{\beta}}{2}\| \bz^{k(\ell)+1} - \bz^k\|,
	\end{equation}
	where we use $k(\ell)+1$ on the superscript of $\bz^{k(\ell)+1}$ to emphasize the dependence of this vector on the stepsize.
	Recall that the iterations conducted by Algorithm~\ref{alg:dist} are simply the agent-wise counterpart of~\eqref{eq:final_dist_it}. Hence, to prove the desired statement, two things are needed. On the one hand, to show that, under an equivalent initialisation, the schemes \eqref{eq:dist_it} and~\eqref{eq:final_dist_it} generate the same  sequence~$(\bz^k)$. Further, that the condition~\eqref{eq:back-cond-alg2} used for the backtracking linesearch in Algorithm~\ref{alg:dist} is equivalent to~\eqref{eq:ls-alg1-mi}. We start with the latter. Indeed, by recalling the definition of $a_i$ and $c_i$ in~\eqref{eq:aici} we have
	\[
	\begin{aligned}
		\sum_{i=1}^n (a_i+& c_i) \leq 0 \\
		&   \Longleftrightarrow 
		\sum_{i=1}^n \left( \tau_{k(\ell)}^2  \|\nabla_x \phi_i (x^{k(\ell)+1}_i,y^{k(\ell)+1}_i)- s^k_{x,i}\|^2 +  \tau_{k(\ell)}^2 \|-\nabla_y \phi_i (x^{k(\ell)+1}_i,y^{k(\ell)+1}_i)- s^k_{y,i}\|^2 \right) \\
		& \quad  \quad \leq  \sum_{i=1}^n \left( \frac{\beta}{4}  \|x^{k(\ell)+1}_i-x^{k}_i\|^2 +  \frac{\beta}{4}  \|y^{k(\ell)+1}_i-y^{k}_i\|^2 \right) \\
		&  \Longleftrightarrow 
		\tau_{k(\ell)} \|B(\bz^{k(\ell)+1} ) - B(\bz^k)\| \leq \frac{\sqrt{\beta}}{2}\| \bz^{k(\ell)+1} - \bz^k\|,
	\end{aligned}
	\]
	as desired.
	
	Now, consider the iterations~\eqref{eq:dist_it} and~\eqref{eq:final_dist_it} initialised with the same choice of $\bz^{-1}$ and $\bz^0$ and with $\bu^0=0$ for~\eqref{eq:dist_it} and $\bv^0=0$ for~\eqref{eq:final_dist_it}. It trivially holds that $\bz^0$ coincides for the two iterations and that $\bv^0 = K \bu^0$. By induction, assume that there exists some $\hat{k}\in\N$ such that these relationships hold for all $k\leq \hat{k}$. We next show that they are preserved for $\hat{k}+1$. 
	
	Multiplying the first equation of~\eqref{eq:dist_it} by $K$ and using the identity $KK = I - W$ from Proposition~\ref{prop:K}\ref{it:K1}, we obtain
	\begin{equation*}
		K\bu^{\hat{k}+1} = K\bu^{\hat{k}} + \tau_{\hat{k}-1}KK\bz^{\hat{k}} = \bv^{\hat{k}} + \tau_{\hat{k}-1}(I-W)\bz^{\hat{k}}.
	\end{equation*}
	This exactly matches the update for $\bv^{\hat{k}+1}$ in~\eqref{eq:final_dist_it}, establishing that $\bv^{\hat{k}+1} = K\bu^{\hat{k}+1}$. Next, we substitute this relation into the primal update. Using $K^* = K$, we have $K^*\bu^{\hat{k}+1} = \bv^{\hat{k}+1}$. Furthermore, the difference term in the second equation of~\eqref{eq:dist_it} becomes
	\begin{equation*}
		K^*(\bu^{\hat{k}+1} - \bu^{\hat{k}}) = \bv^{\hat{k}+1} - \bv^{\hat{k}} = \tau_{\hat{k}-1}(I-W)\bz^{\hat{k}},
	\end{equation*}
	where the last equality is due to the first equation in~\eqref{eq:final_dist_it}. 
	Substituting these into the second equation of~\eqref{eq:dist_it} yields
	\begin{equation*}
		\bw^{\hat{k}+1} = \bz^{\hat{k}} - \tau_{\hat{k}} \bv^{\hat{k}+1} - \tau_{\hat{k}-1}^2(I - W)\bz^{\hat{k}} - \tau_{\hat{k}} B(\bz^{\hat{k}}) - \tau_{\hat{k}-1} \bigl(B(\bz^{\hat{k}})-B(\bz^{\hat{k}-1})\bigr),
	\end{equation*}
	which is exactly the update for $\bw^{\hat{k}+1}$ in~\eqref{eq:final_dist_it}. This finally proves that the sequence $(\bz^k)$ generated by the iterative scheme~\eqref{eq:final_dist_it} is identical to the one generated by~\eqref{eq:dist_it}.

	\ref{item:algsr-3} Furthermore, since we showed above that the linesearch in Algorithm~\ref{alg:dist} is given by a backtracking mechanism, by tailoring the argument in Remark~\ref{remark:backtracking} to this setting, we conclude that the sequence of stepsizes $(\tau_k)$ in Algorithm~\ref{alg:dist} is bounded away from 0 by $\alpha^{\ell}\munderbar{\tau}$, where $\munderbar{\tau}$ is the lower bound imposed on the initial trial stepsizes in Algorithm~\ref{alg:dist}  and $\ell$ is the smallest integer such that $\alpha^{\ell}\munderbar{\tau}\leq\sqrt{\beta}/(2L)$,  with $L$ being the analogous to the Lipschitz  constant considered in Remark~\ref{remark:backtracking}. This shows that Assumption~\ref{ass:tauk} holds.	  
\end{proof}

We finally state our main result on the convergence of Algorithm~\ref{alg:dist} to solutions of Problem~\ref{problem:1}.

\begin{theorem}
	Suppose that Problem~\ref{problem:1} has at least a solution and that the subdifferential sum rules in~\eqref{eq:sumrule} hold. Then the sequence $(\bx^k,\by^k)$ generated by Algorithm~\ref{alg:dist} converges to a point $(\bx^*,\by^*)\in\mathcal{D}_{\R^{n\times p}}\times \mathcal{D}_{\R^{n\times q}}$ such that any pair of rows $(x^*,y^*)$ of $(\bx^*,\by^*)$ solve Problem~\ref{problem:1}. 
    In addition, for the sequences $(\widehat{\bx}^N)$ and $(\widehat{\by}^N)$ defined as in~\eqref{eq:ergodic-squences}, it holds that $(\widehat{\bx}^N,\widehat{\by}^N)\to(\bx^*,\by^*)$ and
    \begin{equation*}
    \quad\left| \Phi(\widehat{\bx}^N,\by^*) - \Phi(\bx^*,\widehat{\by}^N)\right| = \mathcal{O}\left( \frac{1}{N}\right).
\end{equation*}
\end{theorem}

\begin{proof}
	By Proposition~\ref{prop:sp-mi}, the monotone inclusion~\eqref{eq:mm_as_mi} has at least a solution. Then, by Proposition~\ref{prop:dist_ls} and Theorem~\ref{th:conv}, the sequence $(\bx^k,\by^k)$ converges to a point $(\bx^*,\by^*)\in\mathcal{D}_{\R^{n\times p}}\times \mathcal{D}_{\R^{n\times q}}$ solving~\eqref{eq:mm_as_mi}. Finally, by again resorting to Proposition~\ref{prop:sp-mi} the conclusion follows. 
    The last assertion is a direct consequence of Proposition~\ref{prop:dist_ls} and Theorem~\ref{th:first-ergodic}.
\end{proof}

\section{Concluding remarks}\label{sect:conclusion}
	In this work, we proposed a distributed proximal-reflected-gradient method with linesearch (DPRGL) to solve convex-concave min-max problems in the form of Problem~\ref{problem:1}. The key novelty of DPRGL is its ability to handle saddle functions with locally Lipschitz continuous partial gradients without requiring prior knowledge of the local Lipschitz moduli, thereby generalising the framework in~\cite{malitsky2023first}. Our analysis relies on a novel primal-dual forward-reflected splitting algorithm with linesearch for composite monotone inclusion problems involving single-valued, locally Lipschitz continuous operators. A practical limitation of the proposed scheme is that the distributed linesearch implementation relies on a global scalar reduction across all agents (similar to the routines in~\cite{atenas2024distributed}). A promising focus for future research is to design stepsize selection mechanisms that achieve fully decentralised execution without requiring global coordination. This issue has been addressed in different settings by some recent works but under the more restrictive condition of global Lipschitz continuity of the gradients; see, e.g.,~\cite{tam2025decentralisedforwardbackwardtypealgorithmnetworkindependent,MR3999773}.

\section*{Statements and Declarations}

\paragraph{Funding}FA was partially supported by Centro de Modelamiento Matem\'{a}tico (CMM), ACE210010 and FB210005, BASAL funds for center of excellence from ANID-Chile. The research of MND and MKT was supported in part by Australian Research Council grant DP230101749. DTB was partially supported by grant PID2025-168246NB-I00 funded  by
 MICIU/AEI/10.13039/501100011033 and ERDF/EU, as well as by grant CIGE/2026/80 by Conselleria d'Educaci\'o, Cultura i Universitats, Generalitat Valenciana. 

\paragraph{Competing interests} The authors declare no competing interests.

\bibliographystyle{abbrv}
\bibliography{biblio}

\end{document}